\documentclass[11pt, review]{article}
\usepackage{authblk}

\usepackage{amsmath}
\usepackage{amssymb}
\usepackage{amsthm}
\usepackage{geometry}

\usepackage{url}
\usepackage{enumitem}
\usepackage[utf8]{inputenc}
\usepackage[english]{babel}
\usepackage{longtable}
\usepackage[dvipsnames,table,longtable]{xcolor}
\usepackage{booktabs}
\usepackage[ruled,linesnumbered]{algorithm2e}
\usepackage{bbm}
\usepackage{mathtools}
\usepackage{graphicx}
\usepackage{fix-cm}
\usepackage{subcaption}
\usepackage{pgfplots}
\usepackage[title]{appendix}
\usepackage{multirow}
\usepackage{pdflscape} 
\usepackage{scalerel}
\usepackage{lipsum}
\usepackage{comment}
\usepackage{afterpage}

\usepackage{tikz}
\usetikzlibrary{arrows}
\usetikzlibrary{plotmarks}
\usetikzlibrary{shadings,automata,shadows,calc,arrows}
\usetikzlibrary{decorations.pathmorphing,patterns}
\usetikzlibrary{decorations.pathreplacing}
\usetikzlibrary{decorations.markings}
\usetikzlibrary{calc} % for manimulation of coordinates
\usetikzlibrary{shapes,snakes}
\usetikzlibrary{positioning}
\usetikzlibrary{shapes.geometric}

\usepackage{lipsum}
\makeatletter
\renewcommand\footnoterule{%
  \kern-3\p@
  \hrule\@width \textwidth
  \kern2.6\p@}
\makeatother

\makeatletter
\renewcommand*{\@fnsymbol}[1]{\ensuremath{\ifcase#1\or *\or \dagger\or \ddagger\or **\or
   \mathsection\or \mathparagraph\or \|\or  \dagger\dagger
   \or \ddagger\ddagger \else\@ctrerr\fi}}
\makeatother

\allowdisplaybreaks

\newcommand{\textBF}[1]{%
    \pdfliteral direct {2 Tr 0.5 w} %the second factor is the boldness
     #1%
    \pdfliteral direct {0 Tr 0 w}%
}

\newcommand{\R}{\mathbb{R}}

\newcommand{\Z}{\mathbb{Z}}

\newcommand{\Zplus}{\Z_{+}}

\newcommand{\bsubeq}{\begin{subequations}}
\newcommand{\esubeq}{\end{subequations}}
\newcommand{\BI}{\begin{itemize}}
\newcommand{\EI}{\end{itemize}}
\newcommand{\I}{\item}
\newcommand{\BE}{\begin{enumerate}}
\newcommand{\EE}{\end{enumerate}}

\newcommand{\crevOne}{\color{black}}

\newcommand{\crevNet}{\color{black}}

\newcommand{\crev}{\color{black}}

\newcommand{\cb}{\color{black}}

\newcommand{\degr}{d}
\newcommand{\density}{D}
\newcommand{\neighb}{\mathcal{N}}

\newcommand{\stableset}{\mathcal{SS}}
\newcommand{\clique}{\mathcal{K}}
\newcommand{\deltamiss}{\delta^{miss}}
\newcommand{\set}{\mathcal{S}}

\newcommand{\cliquedigraph}{D}
\newcommand{\orderedcliques}{\mathcal{O}}
\newcommand{\orderclique}{o}
\newcommand{\arcvar}{x}
\newcommand{\initcliquevar}{\gamma}
\newcommand{\lastcliquevar}{\lambda}
\newcommand{\predvar}{p}
\newcommand{\unorderedindvar}{z}
\newcommand{\domain}{\mathcal{D}}
\newcommand{\instance}{(G, K)}
\newcommand{\graph}{G}

\newcommand{\adjmatrix}{A}
\newcommand{\K}{K}
\newcommand{\vertexset}{\mathcal{V}}
\newcommand{\edgeset}{\mathcal{E}}
\newcommand{\arcset}{\mathcal{A}}

\newcommand{\vertexDMDGP}{\mathbb{(CP^{VERTEX})}}
\newcommand{\rankDMDGP}{\mathbb{(CP^{RANK})}}
\newcommand{\combDMDGP}{\mathbb{(CP^{COMBINED})}}
\newcommand{\IPVR}{\mathbb{(IP^{VR})}}
\newcommand{\IPCD}{\mathbb{(IP^{CD})}}
\newcommand{\Unorderedrelax}{\mathbb{(IP^{RELAX})}}

\newcommand{\combCP}{\mathbb{(CP^{COMBINED})}}

\newcommand{\vrvar}{x}
\newcommand{\rankvar}{r}
\newcommand{\vertvar}{v}

\newcommand{\rankind}{r}
\newcommand{\vertind}{v}

\newcommand{\Lorder}{(}
\newcommand{\Rorder}{)}

\newtheorem{ex}{Example}
\newtheorem{definition}{Definition}

\newtheorem{prop}{Proposition}

\newtheorem{pre}{Domain Reduction Rule}

\newtheorem{infeas}{Infeasibility Check}
\newtheorem{sym}{Symmetry Breaking Condition}
\usepackage[colorlinks=true, linkcolor=blue, citecolor=blue]{hyperref}
\title{Constraint Programming Approaches { for} the Discretizable Molecular Distance Geometry Problem}
\author[1]{Moira MacNeil\thanks{m.macneil@mail.utoronto.ca}}
\author[1]{Merve Bodur\thanks{bodur@mie.utoronto.ca}}

\affil[1]{Department of Mechanical and Industrial Engineering, University of Toronto}

\date{}

\begin{document}
	
	\maketitle
%\begin{frontmatter}

	\begin{abstract}
		The Distance Geometry Problem (DGP)  seeks to find positions for a set of points in geometric space when some distances between pairs of these points are known. The so-called discretization assumptions allow us to discretize the search space of DGP instances. In this paper, {\crevOne we focus on a key subclass of DGP, namely	the Discretizable Molecular DGP, and study its associated graph vertex ordering problem, the Contiguous Trilateration Ordering Problem (CTOP), which helps solve DGP.}
		We propose the first constraint programming formulations for {\crevOne CTOP}, as well as a set of checks for proving infeasibility, domain reduction techniques, symmetry breaking constraints and valid inequalities. Our computational results {\crevOne on random and pseudo-protein instances} indicate that our formulations outperform the state-of-the-art integer programming formulations.
	\end{abstract}
	

%\linenumbers

\section{Introduction} \label{sec:intro}

In its essence, Distance Geometry seeks to find positions for a set of points in geometric space when some distances between pairs of these points are known \cite{dgpbook,mucherino2012DG}. This has many applications, including in molecular geometry, where {\crevNet Nuclear Magnetic Resonance (NMR) spectroscopy gives estimates of some interatomic distances and the three-dimensional structure must be determined} \cite{dgpbook}. {\crevOne Here, the points to be positioned are the atoms of a molecule. As seen in Figure  \ref{fig:DGPgraph}, we have a set of five atoms and know some pairwise distances between them. We would like to give the atoms coordinates in Euclidean space so that the distances are preserved; a possible three-dimensional mapping of these atoms is pictured in Figure \ref{fig:DGPrealize}.} In wireless sensor localization, the positions of wireless sensors such as smartphones must be determined using the estimated distance between sensors, but there is also a fixed component of the network such as routers \cite{liberti2014euclidean}. Other applications include astronomy, robotics, statics and graph rigidity, graph drawing, and clock synchronization \cite{dgpbook, mucherino2012DG, liberti2014euclidean}. More recent applications have arisen from a new variant of the Distance Geometry Problem (DGP), namely dynamical DGP, including air traffic control, crowd simulation, multi-robot formation, and human motion retargeting  \cite{mucherino2012DG}.

% \begin{tikzpicture}[main_node/.style={circle,fill=white!80,draw,inner sep=0pt, minimum size=16pt},
%	line width=1.2pt]
%	\node[main_node] (v1) at (3,0) {$1$};
%	\node[main_node] (v2) at (0,0) {$2$};
%	\node[main_node] (v3) at (0,3) {$3$};
%	\node[main_node] (v4) at (3,3) {$4$};
%	\node[main_node] (v5) at (5,1.5) {$5$};
%	\draw[dashed,  line width=.8] (v1) -- node[below, blue] {\footnotesize 2} (v2);
%	\draw[dashed,  line width=.8] (v1) -- node[above, blue, near end, xshift=0.1cm] {\footnotesize 2.7} (v3);
%	\draw[dashed,  line width=.8] (v1) -- node[right, blue] {\footnotesize 4.5} (v4);
%	\draw[dashed,  line width=.8] (v1) -- node[right, blue, yshift=-0.1cm] {\footnotesize 2.8} (v5);
%	\draw[dashed,  line width=0.8] (v2) -- node[left, blue] {\footnotesize 2.3} (v3);
%	\draw[dashed, line width=.8] (v2) -- node[above, left, xshift=-.6cm, yshift=-.5cm, blue] { \footnotesize 4} (v4);
%	\draw[dashed, line width=.8] (v2) -- node[above, xshift=-1cm, yshift=-.3cm,blue] {\footnotesize 2} (v5);
%	\draw[dashed, line width=.8] (v3) -- node[above, blue] { \footnotesize 5.4} (v4);
%	\draw[dashed, line width=.8] (v4) -- node[right, blue, yshift=0.1cm ] {\footnotesize 4.5} (v5);
%\end{tikzpicture}

{\crevOne \begin{figure}[h]
	\begin{subfigure}{0.5\textwidth}
		\centering
 \begin{tikzpicture}[main_node/.style={circle,fill=white!80,draw,inner sep=0pt, minimum size=16pt},
line width=1.2pt]
\node[main_node] (v1) at (3,0) {$1$};
\node[main_node] (v2) at (0,0) {$2$};
\node[main_node] (v3) at (0,3) {$3$};
\node[main_node] (v4) at (3,3) {$4$};
\node[main_node] (v5) at (5,1.5) {$5$};
\draw[dashed,  line width=.8] (v1) -- node[below, blue] {\footnotesize 2} (v2);
\draw[dashed,  line width=.8] (v1) -- node[above, blue, near end, xshift=0.1cm] {\footnotesize 2.7} (v3);
\draw[dashed,  line width=.8] (v1) -- node[right, blue] {\footnotesize 4.5} (v4);
\draw[dashed,  line width=.8] (v1) -- node[right, blue, yshift=-0.1cm] {\footnotesize 2.8} (v5);
\draw[dashed,  line width=0.8] (v2) -- node[left, blue] {\footnotesize 2.3} (v3);
\draw[dashed, line width=.8] (v2) -- node[above, left, xshift=-.6cm, yshift=-.5cm, blue] { \footnotesize 4} (v4);
\draw[dashed, line width=.8] (v2) -- node[above, xshift=-1cm, yshift=-.3cm,blue] {\footnotesize 2} (v5);
\draw[dashed, line width=.8] (v3) -- node[above, blue] { \footnotesize 5.4} (v4);
\draw[dashed, line width=.8] (v4) -- node[right, blue, yshift=0.1cm ] {\footnotesize 4.5} (v5);
\end{tikzpicture}
		\caption{}
\label{fig:DGPgraph}
	\end{subfigure}
\begin{subfigure}{0.5\textwidth}
\centering
\begin{tikzpicture}[main_node/.style={circle,fill=white,draw,inner sep=0pt, minimum size=16pt},
line width=1.2pt]
\draw [->, gray, line width=1] (0,0,0) -- (4,0,0)  node[right] {x}; 
\draw [<->, gray, line width=1] (0,-2.5,0) -- (0,3,0) node[above] {y}; 
\draw [->, gray, line width=.8] (0,0,0) -- (0,0,5) node[left] {z};  
\node[main_node] (v2) at (1,0,0) {$2$};
\node[main_node] (v1) at (3,0,0) {$1$};
\draw[dashed,  line width=.8] (v1) -- node[above, blue] {\footnotesize 2} (v2);
\node[main_node] (v3) at (1.5,2,-1) {$3$};
\draw[dashed,  line width=.8] (v1) -- node[right, blue] {\footnotesize 2.7} (v3);
	\draw[dashed,  line width=.8] (v2) -- node[right, blue] {\footnotesize 2.3} (v3);
\node[main_node] (v4) at (1,0,4) {$4$};
\draw[dashed,  line width=.8] (v1) -- node[right, below, xshift=.3cm, yshift=.1cm,  blue] {\footnotesize 4.5} (v4);
	\draw[dashed,  line width=.8] (v2) -- node[right, blue] {\footnotesize 4} (v4);
	\draw[dashed,  line width=.8] (v3) -- node[above, left, blue] {\footnotesize 5.4} (v4);
\node[main_node] (v5) at (1,-2,0) {$5$};
\draw[dashed,  line width=.8] (v1) -- node[below, right, blue] {\footnotesize 2.8} (v5);
	\draw[dashed,  line width=.8] (v2) -- node[left, blue, yshift=-.5cm, xshift=0.05cm] {\footnotesize 2} (v5);
	\draw[dashed,  line width=.8] (v4) -- node[below, blue] { \footnotesize4.5} (v5);
\end{tikzpicture}
		\caption{}
\label{fig:DGPrealize}
\end{subfigure}
\caption{\crevOne(\subref{fig:DGPgraph}) A set of points, $1,2,\hdots,5$, and known pairwise distances. (\subref{fig:DGPrealize}) An embedding of these points in three dimensions.}
\label{fig:DGPexample}
\end{figure}
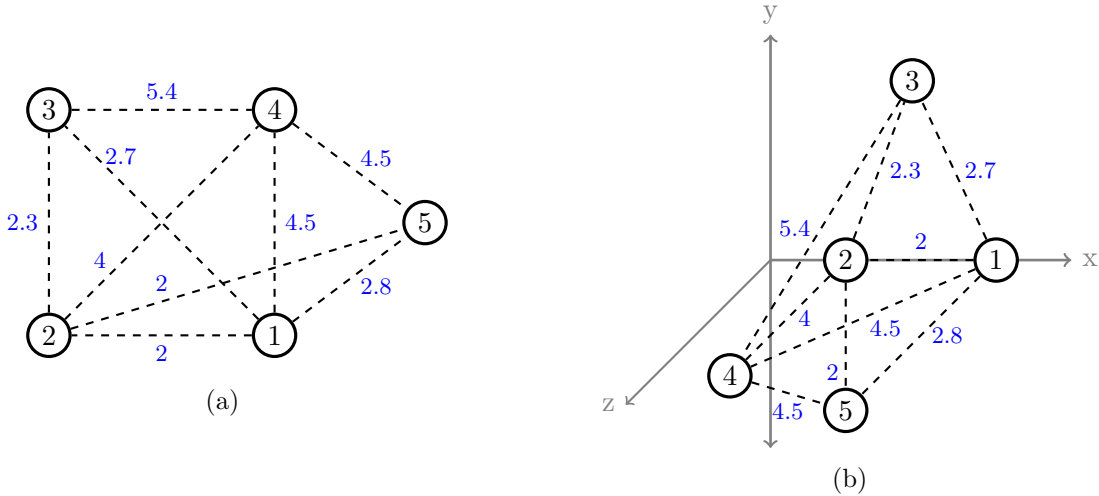
}

 The input to the DGP can be represented as a graph, say $G$, where the vertices are the points we would like to position and weighted edges represent known distances between two points. Formally, we give the definition from \cite{dgpbook}.
 \begin{definition}[Distance Geometry Problem] \label{def:DGP}
 	Given an integer $\K > 0$, and a simple, undirected graph $\graph = (\vertexset,\edgeset)$ with edge weights $w  : \edgeset \to (0, \infty)$,
 	find a function $x : \vertexset \to  \R^\K$ such that for all $\{u,v\} \in \edgeset$:
 	\[
 	\|x(u) - x(v)\| = w(u,v).
 	\]
 \end{definition}
 
If this function $x$ exists, it is called a \textit{realization} for $\graph$, we also refer to the realization as an \textit{embedding} of $\graph$. {\crevNet For example, Figure \ref{fig:DGPrealize} shows a realization with $\K=3$ for the graph in Figure \ref{fig:DGPgraph}, which does not have a realization when $\K=2$.} We assume $\graph$ is connected, since determining if a disconnected graph has a valid realization is equivalent to determining if its connected components have a realization \cite{cassioli2015}. We also assume the norm in Definition \ref{def:DGP} is the Euclidean norm for the remainder of this paper, however this need not be the case. We also note that there exists a form of the problem where the function need not satisfy the strict equality in Definition \ref{def:DGP}, but rather must {\crevOne satisfy $w_l(u,v) \leq \|x(u) - x(v)\|  \leq w_u(u,v)$ for {\crevNet lower and upper} bounds on the weights $w_l(u,v) $ and $w_u(u,v)$ respectively;} this is called the interval DGP \cite{Goncalves2017,Lavor2013}.

 The DGP {\crev with integer weights} is $\mathcal{NP}$-Complete for $\K=1$ and $\mathcal{NP}$-Hard for $\K>1$ \cite{saxe1980embeddability}, motivating the need for solution methods that are able to solve the problem in practice. Solution methods for the DGP include nonlinear programming, semi-definite programming, and the geometric build-up methods  \cite{mucherino2012DG, liberti2014euclidean, mucherinoHDR,Krislock2010}.  If the distances between all pairs of vertices in $\graph$ are known, i.e., $\graph$ is complete, and we assume a solution exists in $\R^\K$, there is a procedure for finding the realization by solving a series of linear equations \cite[{\crev Chapter 3}]{dgpbook}. 
 
 In most applications, the input graph is not complete, {\crevOne i.e., not all pairwise distances are known. In such a case,  there exists an iterative procedure to position the vertices. In this procedure, we begin by choosing $\K$ vertices for which all pairwise distances are known; without loss of generality, we may position them at any coordinates which satisfy all their pairwise distances. Next, we choose a vertex with enough pairwise distances (at least $\K$) relative to the previously positioned vertices and fix its position using this information. We repeat, fixing vertices with at least $\K$ pairwise distances to the previously fixed vertices, until we have a realization. This procedure is the idea behind the branch-and-prune (BP) algorithm, whose success depends on choosing the right fixing order \cite{dgpbook, liberti2014euclidean, mucherinoHDR, Mucherino2012}. Thus we first aim to find the \emph{order} in which the BP algorithm should iterate over vertices \cite{lavor2012discretizable}. The class of {\crevNet DGP} instances possessing such an order {\crev is} called Discretizable DGPs (DDGPs). Otherwise, there may still be a realization, but the instance does not have enough distance information to determine one using the BP algorithm; as such it could be more difficult to find. We note that the input to these discretized ordering problems is an unweighted graph, and the output is a vertex order which is given as input together with the actual pairwise distances to the BP algorithm to find a realization. {\crevNet We direct the reader to \cite{dgpbook, liberti2014euclidean} for a formal definition and more detailed explanation of DDGPs and the BP algorithm.} 
% 	in these problems we do not need the actual distance information, we need only the existence of such information and as such knowing pairwise distances is equivalent to the existence of an edge between two vertices \cite{dgpbook}.
}

{\crevOne  In this paper, we study a key subclass of DDGPs, namely the Discretizable Molecular Distance Geometry Problem (DMDGP) \cite{dgpbook} {for which the literature is quite limited}, {\crevNet as compared to other DDGPs}. The ordering problem associated with DMDGP is the Contiguous Trilateration Ordering Problem (CTOP). Taking an unweighted graph as input, CTOP searches for an ordering of the vertices that ensures the first $\K$ vertices are pairwise adjacent and each subsequent vertex in the order is adjacent to the $\K$ \emph{immediately} previously ranked vertices. {\crevNet The motivation for solving CTOP, instead of a relaxation of the problem, comes from the key application of realizing protein backbones with NMR data, since a certain structure is required to compute interatomic distances using covalent bond lengths and angles between them. These orders also have desirable properties: an order that is a solution to CTOP exhibits symmetries in the BP tree, so to find all solutions for a DMDGP the BP algorithm only needs to be applied to find a single solution and the rest can be enumerated using symmetry \cite[Chapter 5]{dgpbook}.}
  There are other ordering problems associated with DDGPs, one of the most commonly studied being the Discretizable Vertex Ordering Problem (DVOP) which relaxes the need for vertices to be adjacent to the immediately previous vertices to \emph{any} vertex of lower rank. {\crevNet We remark that the solution to CTOP is a solution to DVOP but the converse is not true.}
  
 DVOP has been shown to be polynomially solvable via a greedy algorithm in fixed dimension \cite{Lavor2012}. Moreover, integer programming methods have been used to find DVOP orders which are optimal with respect to some measure of their BP tree size \cite{Omer2017}, and algorithms finding partial orders that optimize the BP search space have been proposed \cite{Goncalves_PartialOrders}. In contrast to DVOP, CTOP is shown to be $\mathcal{NP}$-complete and several integer programming formulations have been proposed to solve it on small instances \cite{cassioli2015}.
 {\crevNet In \cite{mucherino2015deBruijn}, the same underlying graph structure as one of the integer programming formulations of \cite{cassioli2015} is used and an order is constructed using a path in this graph. The algorithm is applied directly to small protein instances and the authors note it cannot be extended to large molecules. As such, we consider \cite{cassioli2015}  the only work of comparison for our study.}

  {\crev The BP algorithm was first proposed for CTOP in \cite{lavor2012discretizable}, {\crevNet and in \cite{Goncalves2016}  answer set programming was shown to improve the performance of the BP algorithm for CTOP.} Furthermore, DMDGP vertex orders with repeated vertices, called re-orders, are considered \cite{Lavor2013}, their computational complexity is analyzed \cite{cassioli2015} and further studied in detail \cite{LAVOR2019}.  }

The rest of the paper is organized as follows. In Section \ref{sec:prelim}, we present in detail the {\crevOne CTOP} and review two existing integer programming (IP) formulations from the literature. In Section \ref{sec:CPmodels}, we introduce three novel constraint programming (CP) formulations for {\crevOne CTOP}. We then present,  in Section \ref{sec:enhance}, a {\cb structural study of CTOP which may eventually aid in its solution}. %series of enhancements which may aid in the solution of {\crevOne CTOP}, in Section \ref{sec:enhance}. 
Finally, in Section \ref{sec:results}, we present a computational study which compares the CP and IP models, and which demonstrates {\cb the current utility and drawbacks of the structural findings}.
%the utility of the enhancements. 

We note that an overview of our paper, namely the models from the literature as well as our proposed models and structural ideas, is provided in Table \ref{table:DMDGPsummary} of {\crevNet Appendix}  \ref{app:summary}.

% 	{\crevOne it can be shown the solution set is finite if we have a realization for $\K$ vertices of the instance and every other vertex,  $v \in \vertexset$, has edges $\{v, i\}, \{v, j\}, \{v, l\} \in \edgeset$, where the positions of $i,j,l \in \vertexset$ have already been fixed \cite{dgpbook, liberti2014euclidean, mucherinoHDR, Lavor2012, Mucherino2012}. This search space induces a binary tree structure where each layer of the tree enumerates all possible positions for a fixed vertex, where a realization is a path in the tree from the root to a leaf \cite{mucherinoHDR}. The number of possible positions for a vertex can be pruned if the graph contains more edges than those that satisfy the discretization conditions. The branch-and-prune (BP) algorithm gives the possible positions of vertices, relative to the previously placed vertices, one at a time, eliminating a position if there is an edge between the current vertex and previous vertices that makes the position infeasible \cite{dgpbook, liberti2014euclidean, mucherinoHDR}.

\section{Preliminaries} \label{sec:prelim}
In this section, we introduce the common notation used in the paper, provide the problem definition and briefly present the existing formulations for the problem.
% ------------------------- Notation ------------------------- %
%{\bf Notation.} 
\subsection{Notation}
All sets are denoted calligraphically. Let $\graph = (\vertexset, \edgeset)$ be an undirected graph, where $\vertexset$ is the set of vertices and $\edgeset$ is the set of edges. The \emph{adjacency matrix} of $\graph$ is denoted by $\adjmatrix$, i.e.,  $\adjmatrix_{v,u} = 1 $ if and only if edge $\{u,v\} \in \edgeset$. Denote the \emph{neighbourhood} of a vertex $\vertind$ as $\neighb(\vertind)$,  i.e., $\neighb(\vertind) = \{ u \in \vertexset : \{u,v\} \in \edgeset\}$. Thus $\vertind \notin \neighb(\vertind)$ and the \emph{degree} of $\vertind$ is $\degr(\vertind) = |\neighb(\vertind)|$. We let $\graph[\vertexset'] = (\vertexset', \edgeset')$ be the subgraph of $\graph$ \emph{induced} by $\vertexset' \subseteq \vertexset$, and thus $\edgeset' = \{\{u,v\}  \in \edgeset : u,v \in \vertexset' \}$. A \emph{clique}, $\clique$, in $\graph$ is a set of vertices $\{v_1, v_2, \hdots, v_{|\clique|}\} \subseteq \vertexset$ such that $\{v_i, v_j\} \in \edgeset$ for all $v_i, v_j \in \clique$ such that $v_i \neq v_j$. Similarly, a \emph{stable set}, $\stableset$, in $\graph$ is a set of vertices $\{u_1, \hdots, u_{|\stableset|}\} \subseteq \vertexset$ such that $\{u_i, u_j\} \notin \edgeset$ for all $u_i, u_j \in \stableset$. We define an \emph{adjacent predecessor} of a vertex $v \in \vertexset$ as $u \in \vertexset$ with $\{u,v\} \in \edgeset$ such that $u$ precedes $v$ in a vertex order, and we define {\crevOne a \emph{contiguous predecessor} of a vertex $v \in \vertexset$ as $u \in \vertexset$ with $\{u,v\} \in \edgeset$, such that there is no $w \in \vertexset$ with $\{w,v\} \notin \edgeset$ between $v$ and $u$ in the vertex order. Thus, $u$ is an adjacent predecessor that immediately precedes $v$ in the {\crevNet order, meaning} there is no non-adjacent vertex between $u$ and $v$ in the vertex order.}

For $a, b \in \Zplus$, $a \leq b$,  we introduce the notation $[a] = \{0, 1, \hdots, {\crevOne a}\}$ and $[a, b] = \{a, a+1, \hdots, b\}$. If $a>b$, then  $[a,b]=\emptyset$, similarly if $a<0$, then  $[a]=\emptyset$. Indices follow these conventions: indices start at $0$, so that the possible positions of a vertex order are {\crev $[|\vertexset|-1]$}. We let $|\vertexset| = n$, and use $|\vertexset|$ in relation to vertices and $n$ in relation to ranks{\crevOne , i.e., positions,} of a vertex order. 

Finally, we introduce the set  $\vertexset^{\degr[\K, \K+\delta]} = \{\vertind \in \vertexset : \ \degr(\vertind) \in [\K,  \K+ \delta]\}$ for some fixed positive integer $\delta$, the set of vertices with degrees in $[\K, \K+\delta]$.

% ------------------------- Problem Definition ------------------------- %
\subsection{Problem Definition}
{\crevOne Given an integer dimension $\K$,} the \textit{Discretizable Molecular Distance Geometry Problem} (DMDGP) \cite{cassioli2015} is the {\crev problem of finding a realization in $\R^\K$ given a} simple, connected, undirected graph instance, $\graph = (\vertexset,\edgeset)$, possessing a total order of its vertices that satisfies the following:
\BE[label=(\roman*)]
\item the first $\K$ vertices in the order form a clique in the input graph $\graph$, and
\item {\crevOne each vertex with rank $\geq \K$ has at least $\K$ \textit{contiguous predecessors}. That is, for each vertex at position $k \in [\K, n-1 ]$ along with the vertices at positions $[k-\K, k-1]$ form a $(\K+1)$-clique in the input graph. }
\EE
We refer to a total order that satisfies (i) and (ii) as a \emph{DMDGP order}, and a clique satisfying (i) as the \textit{initial clique}. We say an instance for which a DMDGP order exists is \emph{feasible}, otherwise it is infeasible. The problem of determining whether a DMDGP order exists for $\graph$ is known as the Contiguous Trilateration Ordering Problem (CTOP) \cite{cassioli2015}. An instance of CTOP, i.e., an integer $\K > 0$ and a simple, undirected, connected graph $\graph = (\vertexset,\edgeset)$, will be denoted $(\graph = (\vertexset,\edgeset), \K)$ or simply $\instance$. Cassioli et al.\ \cite{cassioli2015} proved CTOP is $\mathcal{NP}$-complete. 

{\crevOne We can restate points (i) and (ii) of the definition using a key property of the problem structure, as mentioned in \cite{cassioli2015}:
	\begin{definition} [DMDGP Order]	\label{prop:key}
		Given an integer dimension $\K$, a DMDGP order is a total order of the vertices of a simple, connected, undirected graph $\graph = (\vertexset,\edgeset)$, so that the order forms a series of $(\K+1)$-cliques which overlap by at least $\K$ vertices. 
	\end{definition}
We note that there may be extra edges between vertices than those that form the overlapping cliques. This important property is depicted in Figure \ref{fig:DMDGPcliques} for Example \ref{ex:DMDGP}. Considering this overlapping clique structure, we refer to the first of those overlapping cliques as the \emph{initial $(K+1)$-clique}.}

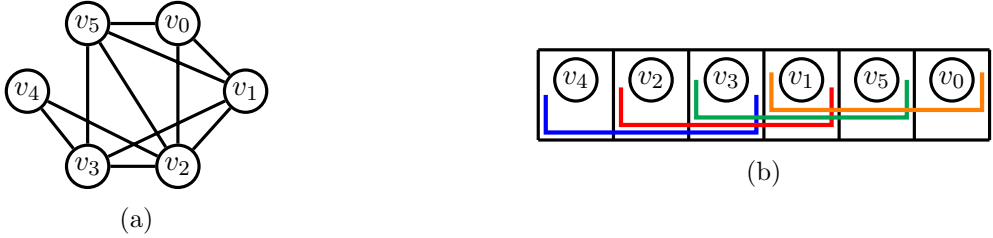
\begin{figure}[h]
	\begin{subfigure}{0.5\textwidth}
		\centering
		\begin{tikzpicture}[main_node/.style={circle,fill=white!80,draw,inner sep=0pt, minimum size=16pt},
		line width=1.2pt]
		\node[main_node] (v0) at (1.2,1.9) {$v_0$};
		\node[main_node] (v1) at (2.1,1) {$v_1$};
		\node[main_node] (v2) at (1.2,0) {$v_2$};
		\node[main_node] (v3) at (0,0) {$v_3$};
		\node[main_node] (v4) at (-0.8,1) {$v_4$};
		\node[main_node] (v5) at (0, 1.9) {$v_5$};
		\draw[-] (v0) -- (v1) -- (v2) -- (v3) -- (v4) -- (v2) -- (v0) -- (v5) -- (v3) -- (v1) -- (v5) -- (v2);	
		\end{tikzpicture}
		\caption{}
		%\caption{A graph instance which is feasible for DMDGP with $\K =2$.}
		\label{fig:DMDGPgraph}
	\end{subfigure}
	\begin{subfigure}{0.5\textwidth}
		\centering
		\begin{tikzpicture}[main_node/.style={circle,fill=white!80,draw,inner sep=0pt, minimum size=16pt},	line width=1.2pt]
		\tikzset{
			ncbar angle/.initial=90,
			ncbar/.style={
				to path=(\tikztostart)
				-- ($(\tikztostart)!#1!\pgfkeysvalueof{/tikz/ncbar angle}:(\tikztotarget)$)
				-- ($(\tikztotarget)!($(\tikztostart)!#1!\pgfkeysvalueof{/tikz/ncbar angle}:(\tikztotarget)$)!\pgfkeysvalueof{/tikz/ncbar angle}:(\tikztostart)$)
				-- (\tikztotarget)
			},
			ncbar/.default=-0.5cm,
		}
		%to make the grid
		\draw[-] (0,0) to (6,0);
		\draw[-] (0,1.2) to (6,1.2);
		\draw[-] (0,0) to (0,1.2);
		\draw[-] (1,0) to (1,1.2);
		\draw[-] (2,0) to (2,1.2);
		\draw[-] (3,0) to (3,1.2);
		\draw[-] (4,0) to (4,1.2);
		\draw[-] (5,0) to (5,1.2);
		\draw[-] (6,0) to (6,1.2);
		\node[main_node] (v4) at (0.5,0.8) {$v_4$};
		\node[main_node] (v2) at (1.5,0.8) {$v_2$};
		\node[main_node] (v3) at (2.5,0.8) {$v_3$};
		\node[main_node] (v1) at (3.5,0.8) {$v_1$};
		\node[main_node] (v5) at (4.5, 0.8) {$v_5$};
		\node[main_node] (v0) at (5.5,0.8) {$v_0$};
		\draw[ncbar, blue, ultra thick] (0.1,0.6) to (2.9,.6);
		\draw[ncbar, red, ultra thick] (1.1,0.7) to (3.9,.7);
		\draw[ncbar, Green, ultra thick] (2.1,0.8) to (4.9,.8);
		\draw[ncbar, orange, ultra thick] (3.1,0.9) to (5.9,.9);
		
		\end{tikzpicture}
		\caption{}
		%\caption{Overlapping cliques of the order$\Lorder v_4, v_2, v_3, v_1, v_5, v_0 \Rorder$.}
		\label{fig:DMDGPcliques}
	\end{subfigure}
	\caption{(\subref{fig:DMDGPgraph}) A graph instance which is feasible for CTOP with $\K =2$. (\subref{fig:DMDGPcliques}) Overlapping cliques of the order$\Lorder v_4, v_2, v_3, v_1, v_5, v_0 \Rorder$.}
	\label{fig:DMDGPexample}
\end{figure}

\begin{ex} \label{ex:DMDGP}
	The graph given in Figure \ref{fig:DMDGPgraph} with $\K = 2$ is a feasible instance for CTOP. A possible DMDGP order is $\Lorder v_4, v_2, v_3, v_1, v_5, v_0 \Rorder$. Clearly, since they are adjacent $\{v_4, v_2 \}$ form a clique, $v_3$ is adjacent to both of its immediate predecessors:  $v_4, \text{ and }  v_2$, so $\{v_4, v_2, v_3\}$ form a $(\K+1)$-clique in the input graph. Similarly, $v_1$ is adjacent to  $v_2, \text{ and } v_3$, forming a $(\K+1)$-clique in the input graph and so on. 
\end{ex}

\begin{comment}
We observe the following properties of DMDGP orders (whose proofs are easy, thus omitted):
\begin{property}
If $\instance$ has a DMDGP order with $\K=m$, then it has a DMDGP order for $\K = m-1$. 
\label{prop:DMDGPorders1}
\end{property}
%\vspace*{-1cm}
\begin{property}
If $\instance$ does not have a DMDGP order with $\K=m$, then it does not have a DMDGP order for $\K = m+1$.
\label{prop:DMDGPorders2}
\end{property}
%\vspace*{-1cm}
\begin{property}
Let  $\Lorder v_0, v_1, \hdots, v_{n-1} \Rorder$ be a DMDGP order with $\K > 0$, and let $\Lorder v_i, v_{i+1}, \hdots, v_{i+(m-1)} \Rorder$ be a continuous subset of the order with size $m$. Any continuous subset of a DMDGP order is itself a DMDGP order. 
\label{subset}
\end{property}
\end{comment}

% ------------------------- Existing IPs ------------------------- %
\subsection{Existing Integer Programming Models} 
\label{subsec:DMDGP_exist}
Prior to this work, Cassioli et al.\ \cite{cassioli2015} present three integer programming (IP) formulations for CTOP. Below we summarize their properties, while we provide their full details for completeness in {\crevNet Appendix} \ref{app:IPapp}.
\BI
\I The vertex-rank formulation $\IPVR$: They introduce $|\vertexset| \times n$ binary variables indicating vertex-rank assignment. Then, the model contains ($|\vertexset| + n$)-many 1-1 assignment constraints and ($|\vertexset| \times n$)-many clique constraints.
\I The clique digraph formulation $\IPCD$: They enumerate all ordered cliques of size $(\K+1)$ in $G$, define a clique digraph $\cliquedigraph$ with vertices as those ordered cliques and arcs for pairs of cliques that suitably overlap to follow each other in the order (as in Figure \ref{fig:DMDGPcliques}). {\crevNet We note $\cliquedigraph$ is exactly the pseudo de Bruijn graph structure described in \cite{mucherino2015deBruijn}.} Then, the CTOP solution corresponds to a path in $\cliquedigraph$. This IP model has digraph arc variables, first clique and last clique variables, and precedence variables\footnote{In \cite{cassioli2015}, the variables for predecessors are given as $w_{uv}$, we believe this is a typo that the variables are in fact their $y_{uv}$ variables which are our $\predvar_{uv}$ variables in the version provided in {\crevNet Appendix} \ref{app:CDIP}.} for vertices in $G$. 
\I The unordered clique relaxation $\Unorderedrelax$: They relax the strict clique ordering constraints of the clique digraph formulation, and solve this relaxation as a first check for the existence of a DMDGP order. The benefit of this formulation is that it reduces the number of variables in $\IPCD$, because we have reduced the worst case number of vertices in the  $\cliquedigraph$  by a factor of $(\K+1)!$. When a solution to $\Unorderedrelax$ is found, it must be verified as this solution does not necessarily yield a DMDGP order. The verification is a simple check to ensure the linear order solution forms a DMDGP order. The strength in this formulation is that if $\Unorderedrelax$ is infeasible, there is no DMDGP order for the instance.
\EI

Regarding the vertex-rank formulation $\IPVR$, {\crevOne we prove that its LP relaxation is always feasible.}
\begin{prop}
\label{prop:LPrelax}
	The LP relaxation of $\IPVR$ on any instance $\graph=(\vertexset,\edgeset)$ with $K \geq 2$ is feasible.
\end{prop}
{\crevOne \begin{proof}
See {\crevNet Appendix} \ref{app:VRIP}. 
\end{proof}}

This observation can be taken as a sign of the $\IPVR$ model being weak. In fact, we observe in our computational experiments that especially for infeasible instances, a large number of branch-and-bound nodes are processed due to LP relaxations (and cuts driven on them) not being strong enough to prune infeasible branches early on.

On the other hand, we note that the clique digraph model (its relaxation) mostly suffers from the large number of ordered (unordered) cliques, and hits either the time or memory limit in our numerical experiments.

Lastly, we note that as mentioned in \cite{cassioli2015}, $\IPVR$ works better for feasible  instances, while $\IPCD$ works better for infeasible instances. However, none of them  scale well with the size of the input graph, which motivates our work on developing alternative formulations and study of the structure of DMDGP orders.
\section{Constraint Programming Models} \label{sec:CPmodels}

Constraint Programming (CP) is a natural approach to distance geometry ordering problems since{ \crevOne CTOP is a Constraint Satisfiability Problem (CSP) where the optimal solution is one that merely satisfies all constraints \cite{lecoutre2013constraint}. Unlike Constraint Optimization Problems where the optimal solution minimizes or maximizes some objective function subject to the constraints, CSPs do not have an objective function.} CP has been shown to work well for problems with a permutation structure \cite{Smith01dualmodels} and allows the leveraging of global constraints such as AllDifferent. To our knowledge, no CP model for CTOP has ever been proposed. The flexibility of CP allows for three possible formulations for CTOP.

The first formulation follows naturally from $\IPVR$, from \cite{cassioli2015}. We define integer variables $\rankvar_\vertind $ equal to the rank of vertex $\vertind \in \vertexset$ in the order.
\bsubeq
\label{form:rankDMDGP}
\begin{alignat}{2}
\rankDMDGP: \ & \text{ AllDifferent}(\rankvar_0,\rankvar_1, ..., \rankvar_{|\vertexset|-1}) \quad \label{rankDMDGP:alldiff} && \\ 
&|\rankvar_u - \rankvar_v| \geq \K+1  &&  \forall\  u,v  \in  \vertexset \text{ s.t. } u \neq v, \{u,v\} \notin \edgeset \label{rankDMDGP:clique}\\
&\rankvar_{\vertind} \in [n-1] && \forall \ \vertind \in  \vertexset \label{rankDMDGP:dom}
\end{alignat}
\esubeq
Using the global constraint AllDifferent \eqref{rankDMDGP:alldiff}, we enforce that each vertex has a unique rank. Together with the domain constraints, \eqref{rankDMDGP:dom}, this is equivalent to the one-to-one assignment constraints, in $\IPVR$, since each rank has a possible domain of $[n-1]$ and we are enforcing the constraint over all the rank variables which are indexed by the vertices, i.e., $|\vertexset| =  n$ variables. To enforce clique constraints, \eqref{rankDMDGP:clique}, we use the idea that if two vertices do not have an edge between them, they cannot be in the same $(\K+1)$-clique. In other words their ranks must have a difference of at least $\K+1$. This constraint completely models the clique constraints and the predecessor constraints since if their rank difference is $\leq \K $ then vertices $u \text{ and } v $ must be in the same clique which contradicts there being no edge between them. {\crevOne We note constraints \eqref{rankDMDGP:clique} can also be expressed using a minimum distance constraint, however our experiments show that the constraints as stated above had better performance.}

Secondly, we present what is called a dual formulation in CP \cite{Smith01dualmodels}, where the values and variable meanings are swapped. Let integer variable $\vertvar_{\rankind} $ represent the vertex in position $r$ of the order and let $\text{table}(\edgeset)$ denote the table which contains a tuple for each edge in $\edgeset$.
%\bsubeq
%\label{form:vertexDMDGP}
%\begin{alignat}{2}
%\vertexDMDGP: \ & \text{AllDifferent}(\vertvar_{0}, \vertvar_{1}, .., \vertvar_{n-1}) \label{vertexDMDGP:alldiff} \qquad && \\ 
%& A_{\vertvar_{i},\vertvar_{j}} = 1  &&  \forall \ i \in [0, \K-2], j \in [i+1, \K-1] \label{vertexDMDGP:clique}\\
%& A_{\vertvar_{i},\vertvar_{j}} = 1  &&  \forall \ i\in [\K, n-1], j \in [i - \K ,i-1]   \label{vertexDMDGP:pred}\\
%&\vertvar_{\rankind}\in [|\vertexset|-1] && \forall \ \rankind \in [n-1] \label{vertexDMDGP:dom}
%\end{alignat}
%\esubeq
\bsubeq
\label{form:vertexDMDGP}
\begin{alignat}{2}
\vertexDMDGP: \ & \text{AllDifferent}(\vertvar_{0}, \vertvar_{1}, .., \vertvar_{n-1}) \label{vertexDMDGP:alldiff} \qquad && \\ 
&\crevOne ({\vertvar_{i},\vertvar_{j}}) \in \text{table}(\edgeset)  &&  \crevOne \forall \ i\in [0, n-\K-1], j \in [i + 1 ,i+\K] \label{vertexDMDGP:cliquepred}\\
%&\crevOne ({\vertvar_{i},\vertvar_{j}}) \in \text{table}(\edgeset)    &&  \forall \ i\in [\K, n-1], j \in [i - \K ,i-1]   \label{vertexDMDGP:pred}\\)
&\vertvar_{\rankind}\in [|\vertexset|-1] && \forall \ \rankind \in [n-1] \label{vertexDMDGP:dom}
\end{alignat}
\esubeq
In \eqref{vertexDMDGP:alldiff} we enforce that each rank has a unique vertex, again using AllDifferent.  To enforce the clique and predecessor constraints \eqref{vertexDMDGP:cliquepred}, we use the CP notion of \textit{table constraints}.  {\crevNet These are global constraints that can be seen as an extension of domain constraints for a set of decision variables, where only certain predefined value combinations are allowed for those variables. Explicitly capturing some relationships between variables, they lead to improved propagation in the CP solver.
In Constraints \eqref{vertexDMDGP:cliquepred}, in order to ensure the overlapping clique structure required, we enforce that for any pair of ranks that differ by at most $\K$, the vertices assigned to these positions in the order must be adjacent in $\graph$, thus they should correspond to the two  endpoints of an edge in $\graph$.} Finally, \eqref{vertexDMDGP:dom} enforces the domain of the variables.

The last CP model is the result of combining the rank and vertex models into a single model by \textit{channelling} the variables using an \textit{inverse constraint}. It uses the constraints for predecessors and cliques from both formulations. This is useful because redundant constraints may actually help CP solvers perform more inference and discover feasible solutions in a shorter amount of time. Having defined $v$ and $r$ variables as before, the combined model is as follows:
\bsubeq
\label{form:combDMDGP}
\begin{alignat}{2}
\combDMDGP: \ & \crevOne \text{AllDifferent}(\vertvar_{0}, \vertvar_{1}, .., \vertvar_{n-1}) \label{combDMDGP:valldiff}  \\ 
& \crevOne \text{AllDifferent}(\rankvar_0,\rankvar_1, ..., \rankvar_{|\vertexset|-1}) \quad \label{combDMDGP:ralldiff} \\
&|\rankvar_u - \rankvar_v| \geq \K+1 \quad && \forall\  u,v  \in  \vertexset \text{ s.t. } u \neq v, \{u,v\} \notin \edgeset \label{combDMDGP:rclique}\\
&\crevOne ({\vertvar_{i},\vertvar_{j}}) \in \text{table}(\edgeset)  &&  \crevOne \forall \ i\in [0, n-\K-1], j \in [i + 1 ,i+\K] \label{combDMDGP:vcliquepred}\\
&\text{inverse}(\rankvar,\vertvar) && \label{combDMDGP:inv}\\
&\rankvar_{\vertind} \in [n-1] && \forall \ \vertind \in  \vertexset \label{combDMDGP:rdom}\\
&\vertvar_{\rankind}\in [|\vertexset|-1] && \forall \ \rankind \in [n-1] \label{combDMDGP:vdom}
\end{alignat}
\esubeq
In this formulation, the inverse constraint \eqref{combDMDGP:inv} enforces the relation $(r_u = j) \equiv (v_j = u)$, which also makes the AllDifferent constraints in the vertex and rank models redundant. {\crevOne However we include them in the model, as when using the appropriate CP parameters they improve the results.}
%The AllDifferent constraints may be included as redundant constraints in the model, however initial computational results showed they were detrimental thus are omitted hereafter.

\section{\cb Structural Analysis of CTOP} \label{sec:enhance}

In this section we present a {\cb study of the} structure of DMDGP orders that {\cb we believe can lead to improvements for the} formulations presented in Section \ref{sec:CPmodels}. {\cb We present these structural findings as} checks for infeasible instances,  procedures for reducing the domains and breaking symmetries in DMDGP orders, and a class of valid inequalities.

% ------------------------- Infeasibility Checks ------------------------- %
\subsection{Infeasibility Checks}

We begin the discussion of {\cb the structure of} CTOP by introducing some simple checks which will immediately indicate if an instance $\instance$ is infeasible. The first check arises from the fact that every vertex needs at least $\K$ neighbours to be a part of a $(\K+1)$-clique. 

\begin{infeas}[Minimum Degree] \label{inf:mindeg}
	
	Given \instance, if $\exists \ \vertind \in \vertexset \text{ such that } \degr(\vertind) < \K$ then $\graph$ does not have a DMDGP order for $\K$.
	
\end{infeas}

{\crevOne Similarly, it is possible to determine a lower bound on the number of edges in $\graph$.
\begin{prop}
	\label{edge_min}
	Given \instance, the minimum number of edges in $\graph$ to have a DMDGP order is  $\left (|\vertexset| - \frac{1}{2} \right )\K - \frac{1}{2} \K^2$.
\end{prop}
\begin{proof}
	See  \ref{proof_edgemin}.
\end{proof}
Proposition \ref{edge_min} leads to a second check for infeasibility.}

%Similarly, it is possible to determine a lower bound on the number of edges in the graph, $\graph$, required for an instance $\instance$ to have a DMDGP order.

\begin{infeas}[Minimum Edges] \label{inf:minedge}
	Given \instance, if $|\edgeset| < \left (|\vertexset| - \frac{1}{2} \right )\K - \frac{1}{2} \K^2$ then this instance is infeasible.
	
\end{infeas}

\begin{ex}
	The graph in Figure \ref{fig:infeasK3} is infeasible with $\K=2$ and $\K = 3$. For $\K=2$, the instance passes Infeasibility Check \ref{inf:mindeg} as every vertex has at least two neighbours. However Infeasibility Check \ref{inf:minedge} proves it is infeasible as the graph has $|\edgeset| = 8$ and the minimum number of edges for $\K=2$ is $(6-0.5)2 - 0.5(2^2) = 9$. For $\K =3$, we can prove this instance is infeasible using Infeasibility Check \ref{inf:mindeg} since $\degr(v_4) = 2 $.
\end{ex}

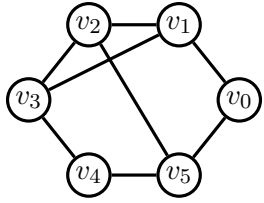
\begin{figure}[h]	
	\centering
	\begin{tikzpicture}[main_node/.style={circle,fill=white!80,draw,inner sep=0pt, minimum size=16pt, scale = 1},
	line width=1.2pt,  scale = 1]
	\node[main_node] (v1) at (3,1) {$v_0$};
	\node[main_node] (v2) at (2.2,2) {$v_1$};
	\node[main_node] (v3) at (1,2) {$v_2$};
	\node[main_node] (v4) at (0.2,1) {$v_3$};
	\node[main_node] (v5) at (1,0) {$v_4$};
	\node[main_node] (v6) at (2.2,0) {$v_5$};
	\draw[-] (v1) --  (v2);
	%	\draw[-] (v1) -- (v3);
	\draw[-] (v1) --  (v6);
	\draw[-] (v2) -- (v3);
	\draw[-] (v2) --  (v4);
	%	\draw[-] (v2) --  (v5);
	%	\draw[-] (v2) --  (v6);
	\draw[-] (v3) -- (v4);
	\draw[-] (v5) --  (v6);
	\draw[-] (v3) --  (v6);
	\draw[-] (v4) --  (v5);
	%	\draw[-] (v4) --  (v6);
	\end{tikzpicture}
	\caption{A graph which is infeasible for CTOP with $\K=2$  and $\K=3$.}\label{fig:infeasK3}
\end{figure}

Cassioli et al.\ \cite{cassioli2015} establish a lower bound on the degree of a vertex, which depends on its position in the order.  If a vertex has degree $\K$ then it can only be placed in the first or last position, since it can only appear in a single $(\K+1)$-clique. If there are more than two vertices with degree exactly $\K$ the instance must be infeasible as there are only two available positions for these vertices. Similarly, there are four positions available for a vertex with degree $\K+1$; ranks $0,1,n-2,n-1$, so if there are more than four vertices with degree $\K+1$, the instance is infeasible. This argument can be extended to the frequency of all vertices of degree strictly less than $2\K$. We formalize the argument of Cassioli et al.\ \cite{cassioli2015} as Infeasibility Checks. We introduce the set $\vertexset^{\degr[\K, \K+\delta]} = \{\vertind \in \vertexset ~| \ \degr(\vertind) \in [\K,  \K+ \delta]\}$ for some fixed positive integer $\delta$, to help express these arguments. We call vertices $\vertind$ with $\degr(\vertind) < 2\K$ \textit{small degree vertices} and vertices $\vertind$ with $\degr(\vertind) \geq 2\K$ \textit{large degree vertices}.

\begin{infeas}[Upper Bound on Small Degree Vertices] \label{inf:UB}
	Given an instance \instance,
	if $\exists \ \delta \in [\K-1]$ such that $\left |\vertexset^{\degr[\K, \K+\delta]} \right | \geq 2(\delta +1) +1$ then this instance is infeasible.
\end{infeas}
\begin{ex}
	Consider the graph in Figure \ref{fig:infeasUB} with $\K=3$. We will see that it is infeasible by  Infeasibility Check \ref{inf:UB}. Note that this instance cannot be proved infeasible by Infeasibility Check \ref{inf:mindeg} but can be proved infeasible by Infeasibility Check \ref{inf:minedge}. 
	Since $\K=3$, we have $\delta \in [1,2]$. First let $\delta = 1$, we have 
	\[\vertexset^{\degr[3, 4]} = \{v_0, v_1, v_2, v_3, v_4\}\]
	and $\left|\vertexset^{\degr[3, 4]}\right| = 5$. The right-hand side of the expression in Infeasibility Check \ref{inf:UB} with $\delta = 1$ is
	\[2(1+1) +1 = 5\]
	and so we are able to say this instance is infeasible. 
\end{ex}

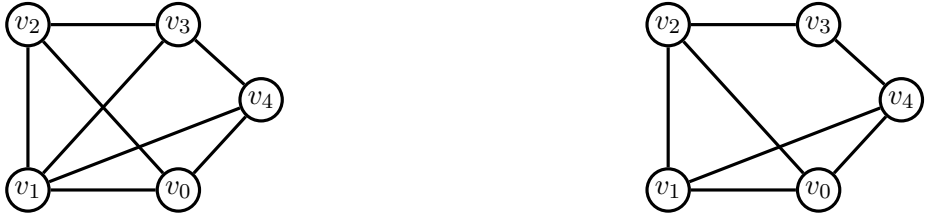
\begin{figure}[h]
	\begin{subfigure}{0.5\textwidth}
		\centering
		\begin{tikzpicture}[main_node/.style={circle,fill=white!80,draw,inner sep=0pt, minimum size=16pt, scale = 1},
		line width=1.2pt, scale = 1]
		\node[main_node] (v1) at (2,0) {$v_0$};
		\node[main_node] (v2) at (0,0) {$v_1$};
		\node[main_node] (v3) at (0,2.2) {$v_2$};
		\node[main_node] (v4) at (2,2.2) {$v_3$};
		\node[main_node] (v5) at (3.1,1.2) {$v_4$};
		\draw[-] (v1) -- (v2);
		\draw[-] (v1) --  (v3);
		%\draw[-] (v1) --  (v4);
		\draw[-] (v1) -- (v5);
		\draw[-] (v2) -- (v3);
		\draw[-] (v2) --  (v4);
		\draw[-] (v2) --  (v5);
		\draw[-] (v3) --(v4);
		\draw[-] (v4) --  (v5);
		\end{tikzpicture}
		\caption{A graph that is infeasible for CTOP with $\K=3$.} \label{fig:infeasUB}
	\end{subfigure}
~
	\begin{subfigure}{0.5\textwidth}
		\centering
			\begin{tikzpicture}[main_node/.style={circle,fill=white!80,draw,inner sep=0pt, minimum size=16pt, scale = 1},
			line width=1.2pt, scale = 1]
			\node[main_node] (v1) at (2,0) {$v_0$};
			\node[main_node] (v2) at (0,0) {$v_1$};
			\node[main_node] (v3) at (0,2.2) {$v_2$};
			\node[main_node] (v4) at (2,2.2) {$v_3$};
			\node[main_node] (v5) at (3.1,1.2) {$v_4$};
			\draw[-] (v1) -- (v2);
			\draw[-] (v1) --  (v3);
			%\draw[-] (v1) --  (v4);
			\draw[-] (v1) -- (v5);
			\draw[-] (v2) -- (v3);
			%	\draw[-] (v2) --  (v4);
			\draw[-] (v2) --  (v5);
			\draw[-] (v3) --(v4);
			\draw[-] (v4) --  (v5);
			\end{tikzpicture}
			\caption{A graph that is infeasible for CTOP with $\K=2$.} \label{fig:infeasLB}
	\end{subfigure}
	\caption{Two graphs that are infeasible for CTOP.}
\end{figure}

Similarly, we have a lower bound on the number of vertices with larger degree. Since the central $(n-2\K)$ vertices in the order are in at least $2\K$ cliques, they must all have degree of at least $2\K$, meaning there must be enough vertices with large degree to occupy these $(n-2\K)$ positions. 

\begin{infeas}[Lower Bound on Large Degree Vertices] \label{inf:LB}
	Given \instance, with $n \geq (2\K+1)$,
	if $\left |\vertexset^{\degr[2\K, n-1]} \right | \leq n-(2\K+1)$ then the instance  is infeasible.
\end{infeas}

\begin{ex}
	Consider the graph in Figure \ref{fig:infeasLB} with $\K=2$. We will see that it is infeasible by  Infeasibility Check \ref{inf:LB}. Note that this instance cannot be proven infeasible by Infeasibility Check \ref{inf:mindeg} or by Infeasibility Check \ref{inf:minedge}. We have $n = 5 = 2\K+1$ and so
	$\vertexset^{\degr[4, 4]} = \emptyset$
	thus  $\left|\vertexset^{\degr[4, 4]}\right| = 0$. We also have $n-2\K-1 =0$,	and so we are able to say this instance is infeasible. 
\end{ex}

\begin{comment}

Next checks establish an upper and lower bound on the number of vertices with small and large degree, respectively, where we call vertices $\vertind$ with $\degr(\vertind) < 2\K$ \textit{small degree vertices} and vertices $\vertind$ with $\degr(\vertind) \geq 2\K$ \textit{large degree vertices}. We introduce the set $\vertexset^{\degr[\K, \K+\delta]} = \{\vertind \in \vertexset | \ \degr(\vertind) \in [\K,  \K+ \delta]\}$ for a given $\delta \in \Z_+$.  

\begin{infeas}[Upper Bound on Small Degree Vertices] \label{inf:UB}
	Given an instance \instance,
	if $\exists \ \delta \in [\K-1]$ such that $\left |\vertexset^{\degr[\K, \K+\delta]} \right | > 2(\delta +1) +1$ then this instance is infeasible.
\end{infeas}

\begin{infeas}[Lower Bound on Large Degree Vertices] \label{inf:LB}
	Given \instance, with $n \geq (2\K+1)$,
	if $\left |\vertexset^{\degr[2\K, n-1]} \right | \leq n-(2\K+1)$ then the instance  is infeasible.
\end{infeas}

\end{comment}

% ------------------------- Domain Reduction ------------------------- %
\subsection{Domain Reduction}

We are able to exploit some structural characteristics of CTOP to help prune variable domains in the CP formulations. Let the domain of an integer variable $x$ be given by $\domain_x$.

First, we extend the lower bounds on the degree of a vertex given by Cassioli et al.\  \cite{cassioli2015} to set the domains for rank variables. As observed previously, a DMDGP order is a series of overlapping cliques of size (at least) $\K+1$. In the minimal case, the first and last vertices in the order are in exactly one clique, the second and second to last vertices are in two cliques, and so on. The central $(|\vertexset|-2\K)$ vertices are in at least $2\K$ cliques. From this we can infer the minimum number of neighbours required by a vertex at a given rank.  

\begin{pre}[Domain Reduction for Small Degree Vertices] \label{red:smalldeg}
	Given an instance \instance, we can define the domain for the rank variables as follows:
	\[\domain_{\rankvar_\vertind} = 
	\begin{cases}
	[\degr(\vertind)-\K] \cup [n-1-(\degr(\vertind)-\K), n-1] & \quad \text{if } \degr(\vertind) < 2\K \\
	[n-1] & \quad \text{otherwise }
	\end{cases}
	\]
\end{pre}

\begin{ex}
	Consider the graph in Figure \ref{fig:DMDGPK2} with $\K =2$. It has two vertices with degree strictly less than $2\K$: $v_0$ and $v_4$. By Domain Reduction Rule \ref{red:smalldeg} we can reduce the domain of both vertices so that their new domains are
	$\domain_{r_{v_{0}}} = \{0,1,4,5\} \ \text{and} \ \domain_{r_{v_{4}}} = \{0,5\}.$
\end{ex}

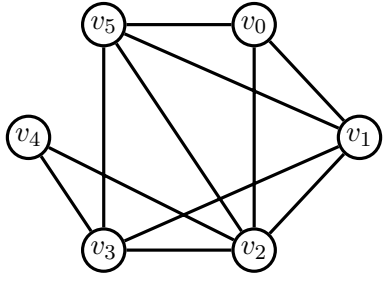
\begin{figure}[h]
	\centering
	\begin{tikzpicture}[main_node/.style={circle,fill=white!80,draw,inner sep=0pt, minimum size=16pt, scale =1},
	line width=1.2pt, scale =1]
	\node[main_node] (v0) at (2,3) {$v_0$};
	\node[main_node] (v1) at (3.4,1.5) {$v_1$};
	\node[main_node] (v2) at (2,0) {$v_2$};
	\node[main_node] (v3) at (0,0) {$v_3$};
	\node[main_node] (v4) at (-1,1.5) {$v_4$};
	\node[main_node] (v5) at (0, 3) {$v_5$};
	\draw[-] (v0) -- (v1) -- (v2) -- (v3) -- (v4) -- (v2) -- (v0) -- (v5) -- (v3) -- (v1) -- (v5) -- (v2);
	\end{tikzpicture}
	\caption{\crevNet A graph instance feasible for CTOP with $\K =2$, to illustrate Domain Reduction Rule 1.}\label{fig:DMDGPK2}
\end{figure}

We are able to extend domain reduction to the vertices that are adjacent to small degree vertices. The intuition is that if a vertex has small degree, the position of its neighbours cannot be too far from that vertex. If the position of a small degree vertex $\vertind^*$ has already been limited, its neighbours must be within the first or the last $ \degr(\vertind^*)$ vertices of the order since they are all connected to $\vertind^*$.

\begin{pre}[Domain Reduction for Neighbourhood of Small Degree Vertices] \label{red:smalldegneighb}
	Given an instance  \instance, with $n \geq (2\K+1)$,  for all $\vertind^* \in \vertexset^{\degr[\K,2\K-1]}$
	
	\[\domain_{\rankvar_\vertind} = [ \degr(\vertind^*)] \cup [n-1-\degr(\vertind^*), n-1] \quad \forall \ \vertind \in \neighb(\vertind^*).
	\]
\end{pre}

\begin{ex}
	Consider the graph in Figure \ref{fig:DMDGPK2_domred}  with $\K =2$. For Domain Reduction Rule \ref{red:smalldegneighb}, we have $ \vertexset^{\degr[2,3]} = \{v_0, v_7\}$. The neighbours of $v_7$ are $v_5$ and {\crevNet $v_4$. Since} they both have degree greater than $2\K$, their domains would not have been reduced by Domain Reduction \ref{red:smalldeg}. We reduce their domains as follows
	\[\domain_{r_{v_{5}}} = \domain_{r_{v_{4}}} =\{0,1,2,5,6,7\}.\]
	For the neighbours of $v_0$, we notice that $[\degr(\vertind^*)] \cup [n-1-\degr(\vertind^*), n-1] = [n-1]$, so we will not be able to reduce their domains.
\end{ex}
\begin{figure}[h]
	\centering
	\begin{tikzpicture}[main_node/.style={circle,fill=white!80,draw,inner sep=0pt, minimum size=16pt, scale =1}, scale =1,
		line width=1.2pt]
		\node[main_node] (v0) at (2,3) {$v_0$};
		\node[main_node] (v1) at (3,1.8) {$v_1$};
		\node[main_node] (v2) at (2,0) {$v_2$};
		\node[main_node] (v3) at (0,0) {$v_3$};
		\node[main_node] (v4) at (-1.3,1) {$v_4$};
		\node[main_node] (v5) at (-1.3, 2.3) {$v_5$};
		\node[main_node] (v6) at (0, 3) {$v_6$};
		\node[main_node] (v7) at (-2.6, 1.6) {$v_7$};
		\draw[-] (v0) -- (v1) -- (v2) -- (v3) -- (v4) -- (v2) -- (v0) -- (v6) -- (v3) -- (v1) -- (v6) -- (v2);
		\draw[-] (v5) -- (v6);
		\draw[-] (v5) -- (v4);
		\draw[-] (v4) -- (v6);
		\draw[-] (v5) -- (v3);
		\draw[-] (v5) -- (v7);
		\draw[-] (v7) -- (v4);
	\end{tikzpicture}
	\caption{\crevNet A graph instance feasible for CTOP with $\K =2$, to illustrate Domain Reduction Rule 2.}\label{fig:DMDGPK2_domred}
\end{figure}
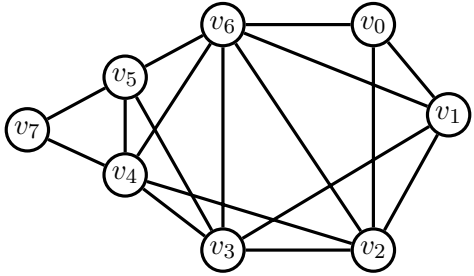

% ------------------------- Symmetry Breaking ------------------------- %
\subsection{Symmetry Breaking}

As observed in \cite{cassioli2015}, reversing a DMDGP order also gives a DMDGP order. We establish that these are not the only symmetries present in DMDGP orders, and present strategies for breaking these symmetries. We begin by a simple condition to break the reverse symmetry. First, notice that if there is a single vertex that has degree $K$ without loss of generality we can fix its position to  $0$. If there is a second vertex with degree $K$ we can fix its position to $n-1$, noting that there are at most two vertices of degree $\K$ in a DMDGP order due to Infeasibility Check \ref{inf:UB}.

\begin{sym} [Degree $K$] \label{sym:degk}
	If $\vertexset^{\degr[\K,\K]} = \{v_i\}$, then let $r_{v_i} =0$. 
	If $\vertexset^{\degr[\K,\K]} = \{v_i, v_j\}$, then let $r_{v_i} =0$ and $r_{v_j} =n-1$.
\end{sym}

%For all other symmetry breaking, if we are trying to impose an order with respect to a vertex in $\vertexset^{\degr[\K,\K]}  $ we must ensure we are not contradicting the position at which it has been fixed. For $r_{v_i} =0$, we ensure we never add a constraint $r_{v_i} > r_{v_j}$ for some $v_j \in \vertexset$, similarly for $r_{v_i} =n-1$, we ensure we never add a constraint $r_{v_i} < r_{v_j}$ for some $v_j \in \vertexset$.

Next, we observe that if two vertices have the same neighbourhood excluding each other, they are interchangeable in the DMDGP order since they will have exactly the same contiguous predecessors. This guarantees a DMDGP order, since the only condition we need to meet to preserve the order if we interchange two vertices is ensuring that they have the appropriate contiguous predecessors. We call this symmetry \textit{pairwise symmetry}, which can be broken by imposing an arbitrary order on the pair of such symmetric vertices. Ideally, we would identify a large set of such vertices and order them. However, identifying such vertex sets can be computationally expensive. We instead identify two types of vertex sets that will allow for easy detection and breaking of pairwise symmetry. Specifically, we consider stable sets and cliques in the input graph.

\begin{sym} [Stable Set] \label{sym:ss}
	For a stable set $\stableset =\{v_1, v_2, \hdots, v_k\} \subseteq \vertexset$ such that $ \neighb(\vertind_i) = \neighb(\vertind_j) \  \forall \ \vertind_i, \vertind_j \in \stableset $  we enforce that $\rankvar_{\vertind_1} < \rankvar_{\vertind_2} < \cdots < \rankvar_{\vertind_k}$.
\end{sym}

\begin{sym} [Clique] \label{sym:clique}
	For a clique  $\clique=\{v_1, v_2, \hdots, v_k\}  \subseteq \vertexset$ such that $ \neighb(\vertind_i) \setminus \clique =\neighb(\vertind_j)\setminus \clique, \  \forall \ \vertind_i, \vertind_j \in \clique $  we enforce that $\rankvar_{\vertind_1} < \rankvar_{\vertind_2} < \cdots < \rankvar_{\vertind_k}$.
\end{sym}

In our experiments, we examine only cliques of size three or less, since we are usually unable to find large cliques satisfying Condition \ref{sym:clique}. Furthermore, we are able to conditionally extend these symmetry breaking conditions to include more vertices. Consider, for example, two vertices $v$ and $u$ whose neighbourhoods differ only by one vertex $w \in \neighb(v)$. If in the DMDGP order $w$ is at least $K+1$ positions away from $v$, the edge connecting them is not necessary to enforce precedence in the order, that is, $w$ is not an contiguous predecessor of $v$ and vice versa. In this case we can essentially consider $u$ and $v$ as having the same neighbourhood and so can impose symmetry breaking on them. For some set $\set \subseteq \vertexset$ we denote $\neighb(\set) = \cup_{v \in \set} \neighb(v) \setminus \set$, the set of all vertices, outside of $\set$ that are adjacent to a vertex in $\set$.

\begin{sym} [Extended Stable Set] \label{sym:ex_Ss}
	Let   $\stableset$ be a stable set meeting Condition \ref{sym:ss} or a single vertex not in any stable set meeting Condition \ref{sym:ss}.	For a vertex $\vertind \in \vertexset \setminus (\stableset \cup \neighb(\stableset))$ such that  $\neighb(\vertind)\setminus \neighb(\stableset) = \{w\}$ we enforce the logical constraints:
	\[|\rankvar_\vertind - \rankvar_w| \geq K+1 \implies \rankvar_\vertind < \rankvar_u \  \forall \  u \in \stableset.\]
	If we have enforced an ordering for $\stableset$ already, we need only add the constraint 
	\[|\rankvar_\vertind - \rankvar_w| \geq K+1 \implies \rankvar_\vertind < \rankvar_{v_1}.\]
\end{sym}

\begin{sym} [Extended Clique] \label{sym:ex_clique}
	Let $\clique$ be a clique meeting Condition \ref{sym:clique} or a single vertex not in any clique meeting Condition \ref{sym:clique}. 	For a vertex $\vertind \in \neighb(\clique)$ such that  $(\neighb(\vertind) \cup \{v\})\setminus (\neighb(\clique) \cup \clique) = \{w\}$ we enforce the logical constraints
	\[|\rankvar_\vertind - \rankvar_w| \geq K+1 \implies \rankvar_\vertind < \rankvar_u \  \forall \  u \in \clique.\]
\end{sym}

Finally, if we have not been able to break any symmetry via any of the previous ways we can arbitrarily choose two vertices and impose an order on them. 

\begin{sym} [Arbitrary]\label{sym:arb}
	For any $v_1, v_2 \in \vertexset$ enforce that $r_{v_1} < r_{v_2}$.
\end{sym}

An example of the strength of symmetry breaking with $\K =2$ can be found in {\crevNet Appendix} \ref{Exapp}.

In our experiments, we {\crev first test for} Symmetry Breaking Conditions \ref{sym:ss} {\crev and \ref{sym:clique} and then for Symmetry Breaking Conditions \ref{sym:ex_Ss} and \ref{sym:ex_clique}, breaking the symmetry where possible}. If none of these conditions are met, we {\crev test for symmetry using}  Symmetry Breaking Condition \ref{sym:degk}, because it is unlikely we will have a vertex with degree exactly $\K$ if $n$ is large. Finally, if all previous Symmetry Breaking Conditions have failed, we {\crev break the symmetry using}   Symmetry Breaking Condition \ref{sym:arb}. 

% ------------------------- Valid Inequalitites ------------------------- %
\subsection{A Class of Valid Inequalities}
\label{sec:validineqs}
Next, we {\cb develop the structure of a class of} valid inequalities for {\cb CTOP. We note that the focus of this study is not the practical application of these valid inequalities but the underlying structural analysis that leads to them}. We proceed with the following intuition: if we identify some subset $\set \subseteq \vertexset$ such that the induced subgraph of $\set$ does not have a DMDGP order, the entire set $\set$ cannot appear consecutively in the order. 

If for a given instance, $\instance$, we are able to identify subsets $\set \subseteq \vertexset$ whose induced graphs, $\graph[\set]$, do not have DMDGP orders for $\K$, we can add cuts to enforce that the difference between the maximum rank and the minimum rank of any element in $\set$ is at least $|\set|$. Let  $r_{max}, r_{min}$ denote the maximum rank and the minimum rank of any vertex  in $\set$, respectively. The valid inequality is
\begin{equation}
\label{eq:scut}
r_{max} - r_{min} \geq |\set|.
\end{equation}

We can improve this cut by examining the vertex in $\set$ with the smallest degree in the induced subgraph. {\crev  Let $\edgeset[\set]$ be the edge  set of $\graph[\set]$, and let  }$\deltamiss_\set(v) = |\{ u \in \set \setminus \{v\}  \ | (v,u) \notin \edgeset[\set]\}|$, i.e., the number of edges with one endpoint at $v \in \set$ missing from $\graph[\set]$ and let $ \deltamiss_\set= \max_{v \in \set} \deltamiss_\set(v)$. If $ \K > |\set| - \deltamiss_\set$, the difference between the maximum rank and the minimum rank must be greater than $ \deltamiss_\set+ \K$, because the $v \in \set $ which has $\deltamiss_\set$, cannot be in a clique with $\deltamiss_\set$ of the vertices in $\set$, so we need at minimum $\deltamiss_\set$ extra vertices between the vertices of $\set$ in the order. Otherwise, if $\K \leq |\set| - \deltamiss_\set $, the difference in ranks must be greater than $|\set|$ which is the inequality \eqref{eq:scut}. So, the valid inequality is
\begin{equation}
\label{scut1}
r_{max} - r_{min}  \geq \max \{ |\set|,  \deltamiss_\set+K\}.
\end{equation}

\begin{ex}
	For the graph in Figure \ref{fig:eg_graph} with $\K=3$, consider $\set = \{v_0, v_3, v_4, v_5\}$ where $\graph[\set]$ has no DMDGP order since $|\edgeset[\set]| = 4 $ and Infeasibility Check \ref{inf:minedge} gives a lower bound of $6$ edges for a DMDGP order.
	We have 
	\[ \deltamiss_\set = \max \{\deltamiss_\set (v_0), \deltamiss_\set (v_3),  \deltamiss_\set (v_4), \deltamiss_\set (v_5) \}\]
	\[ \deltamiss_\set = \max \{2, 0, 1, 1\}\]
	so we have $\max \{ |S|,  \deltamiss_\set  +K\} = \max\{4, 2+3\} = 5$. Thus \eqref{scut1} gives:
	\[r_{max} - r_{min}  \geq 5\]
	which is stronger than the original \eqref{eq:scut},
	\[r_{max} - r_{min}  \geq 4.\]
	Note that this may not have been the case with a different choice of $\set$.
	\begin{figure}[h]
	\centering
	\begin{tikzpicture}[main_node/.style={circle,fill=white!80,draw,inner sep=0pt, minimum size=16pt, scale = 1},
	line width=1.2pt, scale = 1]
	\node[main_node] (a) at (0,0) {$v_0$};
	\node[main_node]  (b) at (1,3) {$v_1$};
	\node[main_node]  (c) at (1.5,1)  {$v_2$};
	\node[main_node]  (d) at (4,0) {$v_3$};
	\node[main_node]  (e) at (4,2) {$v_4$};
	\node[main_node]  (f) at (6,1) {$v_5$};
	\draw[-] (a) -- (b) (a) -- (c) (a) --(d) 
	(b) -- (c) (b) -- (d) (b) -- (e)
	(c) -- (d) (c) -- (e) 
	(d) -- (f) -- (e)
	(b) -- (c) -- (f)
	(d) --(e);
	\end{tikzpicture}
	\caption{Graph with DMDGP order $\Lorder v_5, v_4, v_3, v_2, v_1, v_0 \Rorder$ for $\K=3$. } \label{fig:eg_graph}
\end{figure}
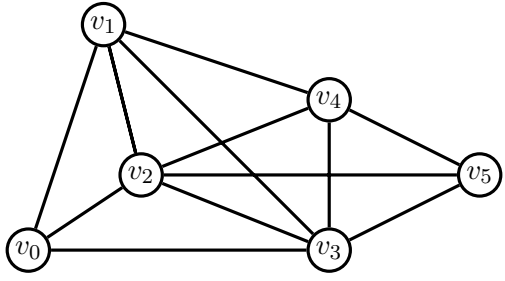
\end{ex}

The task of finding subsets of vertices $\set$ so that the subgraph induced by $\set$ does not have a DMDGP order is as difficult as determining if the whole graph has a DMDGP order. Thus, we would like to find sets of vertices with the most edges missing in their induced subgraph. As the sets with the most missing edges are stable sets, we can consider stable sets in $\graph$  as candidate $\set$ sets. For any stable set $\stableset$, no pair of vertices can appear in the same $(\K+1)$-clique. Thus, each pair of vertices in $\stableset$ needs to have a difference in their ranks of at least $\K+1$, meaning the minimum rank and maximum rank must have a difference of $(|\stableset| -1)(\K + 1)$. The inequality becomes
\begin{equation}
\label{scut2}
r_{max} - r_{min}   \geq (|\stableset| -1)(K+1).
\end{equation}

\begin{ex}
	Using  the graph from Figure \ref{fig:eg_graph} and $\K=3$, consider $\stableset = \{v_0, v_5\}$. The inequality \eqref{scut2} is:
	\[r_{max} - r_{min}  \geq  (2-1)(3+1)  \]
	\[r_{max} - r_{min}  \geq  4 \]
	
\end{ex}

This observation also yields a simple check for infeasibility.
\begin{infeas}
	\label{inf:maxSS}
	Given $\instance$, if the size of the maximum stable set in $\graph$  is greater than $\frac{n}{\K+1} + 1$, we can immediately say  $\graph$ does not have a DMDGP order with $\K$.
\end{infeas}

\begin{prop}
	\label{prop_incompare}
	The inequalities \eqref{scut1} and \eqref{scut2} are incomparable. 
\end{prop}
\begin{proof}
{\crevOne	See \ref{proof_incompare}.}
\end{proof}

\begin{ex}
	Take the wheel graph $W_7$, as pictured in Figure \ref{fig:W7}, with $K=3$. The maximum stable sets are  $\{1,3,5\}$ and $\{2, 4, 6\}$ which have size $3$. The right-hand side of the inequality in Infeasibility Check \ref{inf:maxSS} is $\frac{7}{3+1} + 1 = 2.75 \leq 3$. Thus we can say immediately that this instance is infeasible. 
\end{ex}

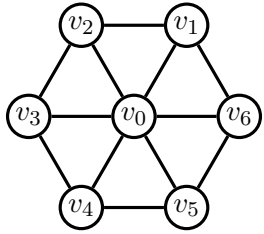
\begin{figure}[h]
	\centering
	\begin{tikzpicture}[main_node/.style={circle,fill=white!80,draw,inner sep=0pt, minimum size=16pt, scale = 1},
	line width=1.2pt, scale = 1]
	\node[main_node] (center) at (0,0) {$v_0$};
	\foreach \phi in {1,...,6}{
		\node[main_node] (v_\phi) at (360/6 * \phi:1.4cm) {$v_\phi$};
		\draw[-] (v_\phi) -- (center);
	}
	\draw[-] (v_1) -- (v_2) -- (v_3) -- (v_4) -- (v_5) -- (v_6) -- (v_1);
	\end{tikzpicture}
	\caption{A wheel graph with seven vertices, $W_7$.} \label{fig:W7}
\end{figure}

In fact, we are able to generalize this for all wheel graphs.

\begin{prop} \label{wheel}
	For any wheel graph, $W_n$, with $\K \geq 2$, if $n$ is odd and $n \geq \frac{\K+1}{\K-1}$ or if $n$ is even and $n \geq 2\frac{\K+1}{\K-1}$ there is no DMDGP order.
\end{prop}
\begin{proof} 
	{\crevOne See \ref{proof_wheel}.}
\end{proof}

Finally, we define these valid inequalities so that they may be added to the $\rankDMDGP$ and $\combDMDGP$ formulations. Given a stable set $\stableset \subseteq \vertexset$, and the rank variables $\rankvar_{\vertind}$ we have {\crevOne two options to implement the valid inequalities, as follows:
\begin{equation} \label{eq:VIabs}
\max\{\rankvar_{\vertind} | \vertind \in \stableset\} - \min\{\rankvar_{\vertind} | \vertind \in \stableset\}\geq (|\stableset|-1)(\K+1)
\end{equation}
\begin{equation} \label{eq:VIMinD}
\text{MinDistance}(\{\rankvar_{\vertind} | \vertind \in \stableset\},\K+1)
\end{equation}
where the constraint \eqref{eq:VIMinD} is a minimum distance constraint which ensures that all variables it acts upon take values at least $\K+1$ apart from each other.}
\section{Computational Results} \label{sec:results}

\subsection{Experimental Setup}
To solve the IPs we use the solver IBM ILOG CPLEX version 12.8.0 and to solve the CPs we use IBM ILOG CP Optimizer version 12.8.0. All models were implemented in C++ and run on MacOs with 16GB RAM and a 2.3 GHz Intel Core i5 processor, using a single thread. The default IP and CP parameters were used with the exception of in $\combDMDGP$ where extended inference was invoked on the AllDifferent constraints. We use $\K=3$ for all experiments as this is the value of $\K$ frequently used in applications. The time limit is set to $7200$ seconds. (We also tested the small random instances with $\K = 4$ and  $\K = 5$; their results are provided in Tables \ref{table:K4} 
and \ref{table:K5} 
of Appendix \ref{app:DMDGPApp}, 
respectively. We note the results were similar to those with $\K=3$.)

\subsection{Instances}
We perform our numerical experiments on a test data set consisting of randomly generated graphs\footnote{Random  instances can be found at \url{https://sites.google.com/site/mervebodr/home/DMDGP_Instances.zip}.} and pseudo-protein graphs\footnote{Pseudo-protein instances can be found at \url{https://gitlab.insa-rennes.fr/Jeremy.Omer/MinDouble}.}.

\paragraph{Random Instances} We divide the random test set into  \emph{small instances}, having $n \in \{ 20, 25, \hdots, 60\}$ and the expected edge density (measured as $\density = \frac{2|\edgeset|}{ n (n-1)} $) in $\{0.3, 0.5, 0.7\}$, and \emph{medium-sized instances} which have $n \in \{65,70 \hdots, 100\}$ and the expected edge density in $\{0.2, 0.3, 0.4, 0.5\}$.  For each $n$, density pair, we generate three graph instances using the dense\_gnm\_random\_graph() method in the NetworkX Python package \cite{networkx}, which chooses a graph uniformly at random from the set of all graphs with $n$ vertices and $m$ edges. {\crevOne These are  Erd\H{o}s-R\'{e}nyi random $G(n,M)$ graphs. } Table \ref{table:DMDGPinst} presents a summary of the instances. We remark that a portion of the instances were unsolved by any method; we denote these as \emph{unsolved} instances.

\begin{table}[h]
	\caption{Random graph instances.}
	\label{table:DMDGPinst}
	\centering
	\small
	\begin{tabular}{c c c c c c c } 
		\toprule
		& $n$ & $D$ & \# Instances &  \# Feasible & \# Infeasible & \# Unsolved \\
		\midrule
		\multirow{3}{*}{small} & \multirow{3}{*}{$20$-$60$} & 0.3 & 27 & 0 & 27 & 0 \\
		&  & 0.5 & 27 & 20 & 7 & 0 \\
		&  & 0.7 & 27 & 27 & 0 & 0 \\[0.5ex]
		\midrule
		Total & && 81 & 47 & 33 & 0\\
		\midrule
		\midrule
		\multirow{4}{*}{medium} & \multirow{4}{*}{$65$-$100$} & 0.2 & 24 & 0 & 24 & 0 \\
		&  & 0.3 & 24 & 0 & 10 & 14 \\
		&  & 0.4 & 24 & 13 & 0 & 11 \\
		&  & 0.5 & 24 & 24 & 0 & 0 \\[0.5ex]
		\midrule
		Total & && 96 & 37 & 34 & 25\\
		\bottomrule
	\end{tabular}
\end{table}
We remark that for the small instance data set, all $27$ graph instances with $D =  0.3$ were infeasible and all $27$ graph instances with $D = 0.7$ were feasible. When $D = 0.5$, we have $7$ instances which are infeasible, all of which have $n=20, n=25$ or $n=30$, the $20$ feasible instances have $n \geq 30$. For medium-sized instances, preliminary results showed all instances with $D > 0.5$ were feasible and solved in less than a second. For this reason we focus our study on instances with  $D \leq 0.5$, which give more insights into the solution methods. The $24$ instances with  $D = 0.2$ were infeasible and the $24$ instances with $D = 0.5$ were feasible. The instances with $D = 0.3$ and $D = 0.4$ were more difficult to solve. For instances with $D = 0.3$, we found $10$ infeasible instances, those with $n = 65,70,75,80$, with all others unsolved. For medium-sized instances with $D = 0.4$, $13$ instances were feasible with $11$ others unsolved. 

{\crevOne \paragraph{Pseudo-Protein  Instances} The pseudo-protein instances are similar to those used in \cite{Omer2017} and were provided by the authors of that paper. They are generated using existing instances of the Protein Data Bank, which are cut to the required number of vertices and then reduced in density by randomly removing edges. This test data set has $209$ instances with vertices in $\{30,40,50,60\}$ and expected edge density in $[0.03,0.12]$. We remark of these pseudo-protein instances $192$ are infeasible and $17$ are feasible.}

\subsection{Computational Results and Discussion}
In this section, we provide our observations based on a thorough computational study. All the detailed experimental result tables are provided in {\crevNet Appendix \ref{app:DMDGPApp}}. 
Here, we summarize our main findings by first outlining the results, without the addition of improvements from Section \ref{sec:enhance}, of the novel CP formulations as compared to the IP formulations from the literature. Then in Sections  \ref{sec:struct} and	\ref{sec:vi_comp} we present a preliminary study of the structural findings and valid inequalities, respectively. 

We begin by noting that the IP formulations from the literature do not perform well against the CP formulations, which can be seen in detail in Table \ref{table:smallExact} of Appendix \ref{app:DMDGPApp}. 
$\IPCD$ is able to solve instances with $D = 0.3$ and $n \leq 40$ in less than a second. These are also the instances that are infeasible. However, for higher densities, and as $n$ increases, $\IPCD$ either hits the time limit or memory limit with $50\%$ of the instances hitting the time or memory limit for $\IPCD$, due to the large number of ordered cliques for larger and more dense instances. On average the instances have $165,875$ ordered cliques, with the smallest number of ordered cliques being $48$ and the largest being $1,406,256$ cliques.  Table 
 {\crevNet \ref{table:CDrelax} 
 	of Appendix \ref{app:DMDGPApp}} 
 shows that $\Unorderedrelax$ has performance similar to $\IPCD$, performing best when the number of vertices is small and when the density is low. For $n \geq 30$,  $\Unorderedrelax$ is unable to solve any feasible instance (all of which have $D = 0.5$ or $D = 0.7$) without running out of time or memory. In fact, when the time limit is hit, we are still in the root node of the tree, for instances with $n\geq 50$. Even when we relax the ordering constraints on the cliques, we still have a large number of unordered cliques. The smallest number of unordered cliques for an instance is $2$ and the largest is $58,594$, on average an instance has $6,911$ unordered cliques. $\IPVR$ begins to hit the time limit at $n = 25$ and for $n \geq 35$ it is only able to solve one instance with $D < 0.7$. This confirms the observations of \cite{cassioli2015}, but we have also shown that neither $\IPCD$ nor $\IPVR$ scale well.  
 
We next compare our CP formulations with the vertex-rank IP formulation $\IPVR$ of \cite{cassioli2015}. In Figure \ref{fig:DMDGPResults}, we provide performance profiles for the solutions times of different models on small instances. Note that solution times are given in a logarithmic scale. We observe that the CP formulations all outperform $\IPVR$. For small feasible instances, as seen in Figure \ref{fig:smallFeas}, $\rankDMDGP$ is able to solve $29$ instances in less than a second. However, for feasible instances  $\combDMDGP$ and $\vertexDMDGP$ perform the best overall, with $\vertexDMDGP$ solving one more instances than $\combDMDGP$ within the time limit. The performance profile for small infeasible instances in Figure \ref{fig:SmallInfeas} shows that $\combDMDGP$ has the best performance on these instances.

\begin{figure}[h]
\hspace*{-0.7cm} 
	\begin{subfigure}{0.5\textwidth}
		\includegraphics[scale=0.23]{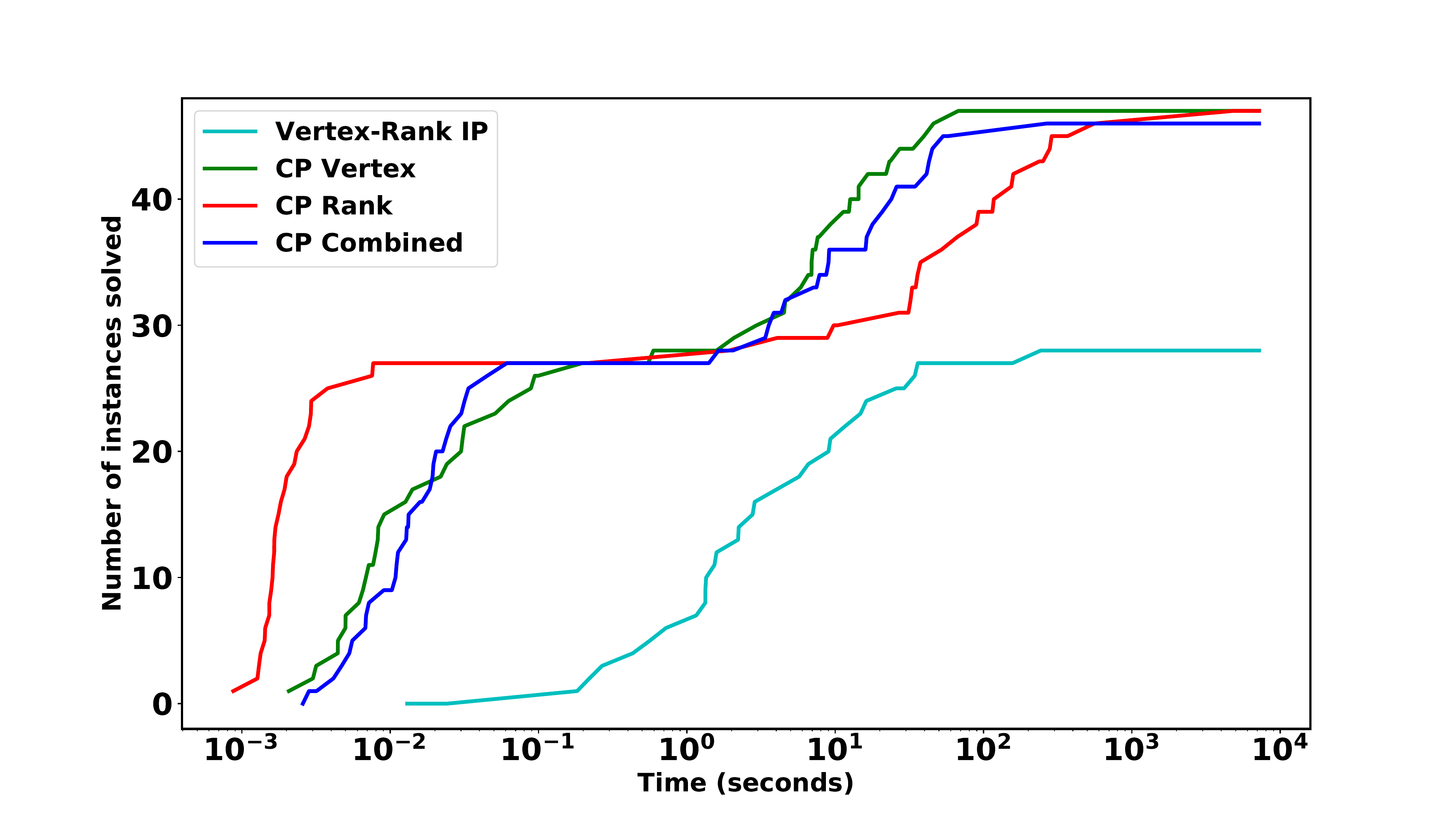}
		\caption{Small feasible instances}
		\label{fig:smallFeas}
	\end{subfigure}
\hspace*{0.1cm} 
	\begin{subfigure}{0.45\textwidth}
		\includegraphics[scale=0.23]{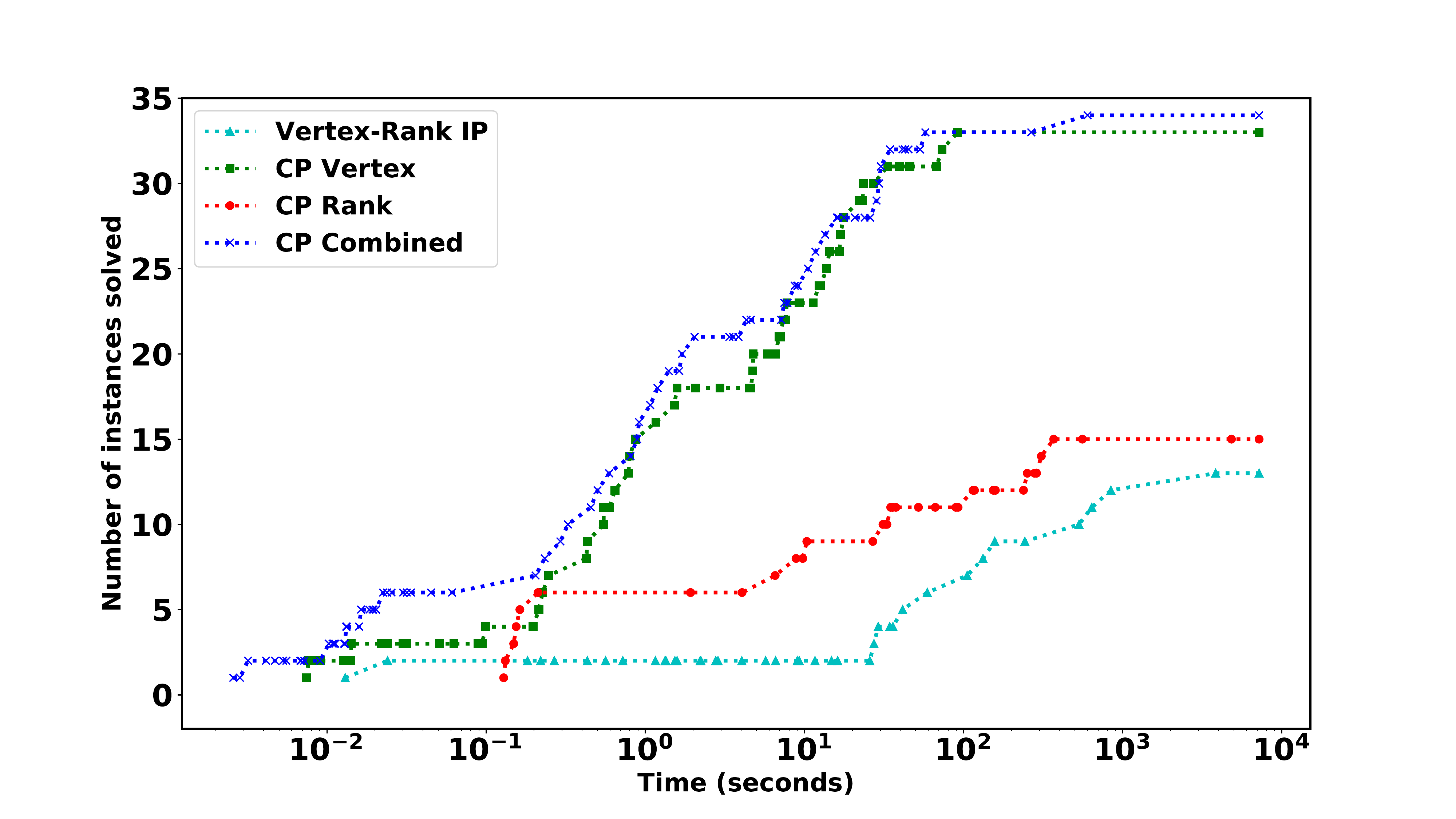}
		\caption{Small infeasible instances}
		\label{fig:SmallInfeas}
	\end{subfigure}
	\caption{Solution times of the models for small instances.}
	\label{fig:DMDGPResults}
\end{figure}

 Table 
 {\crevNet \ref{table:smallExact} 
% 	of the Online Supplement
 of	Appendix \ref{app:DMDGPApp}} 
 also reveals that $\rankDMDGP$ is able to solve instances with  $D = 0.7$ in less than a second, however, it begins to hit the time limit for $n \geq 35$ when $D = 0.3$. For $D = 0.5$, $\rankDMDGP$ is able to solve instances but is {\crevOne almost always} outperformed by $\combDMDGP$ and $\vertexDMDGP$. We also remark that after 25 vertices, {\crevOne with $D=0.5$}, the number of choice points for solving with $\rankDMDGP$ exceed one million {\crevOne for all but two instances}. %$\vertexDMDGP$ performs best on instances with $D = 0.7$, but as $n$ increases, $\vertexDMDGP$ is outperformed slightly by $\combDMDGP$ in terms of time and choice points.

Overall, we conclude lower density instances are more challenging for the CP models. However, high density ($D = 0.7$) is trivial even with $60$ vertices. For these reasons, we focus on densities less than or equal to $0.5$, but increase the {\crevOne number of densities considered}. Due to their poor performance, we exclude the IPs from further study. We now direct our focus to the CP models for medium-sized instances.

\begin{figure}[h]
\hspace*{-0.7cm} 
	\begin{subfigure}{0.5\textwidth}
		\includegraphics[scale=0.23]{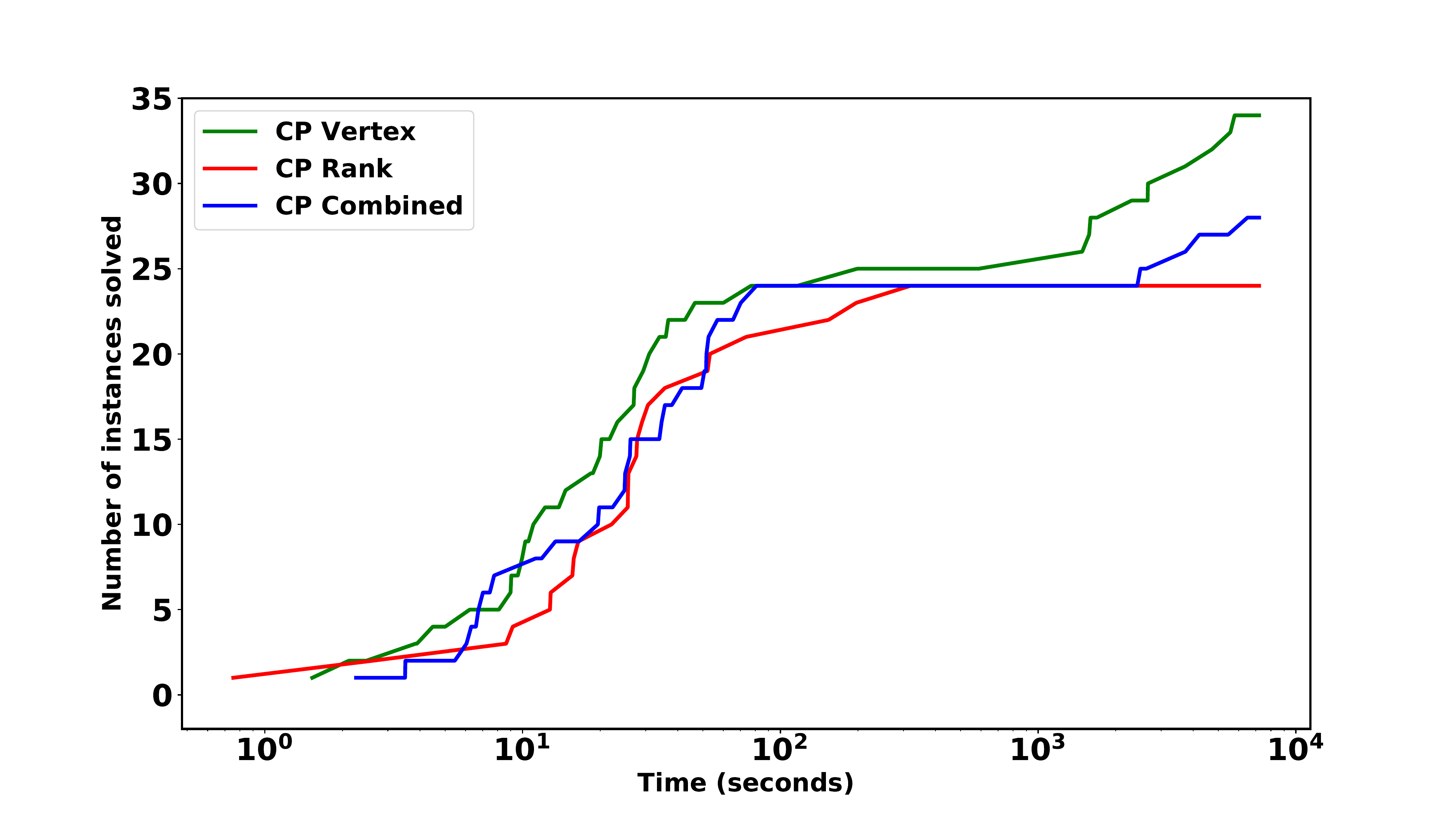}
		\caption{Medium-sized feasible instances}
		\label{fig:LargeFeas}
	\end{subfigure}
\hspace*{0.1cm} 	
	\begin{subfigure}{0.45\textwidth}
		\includegraphics[scale=0.23]{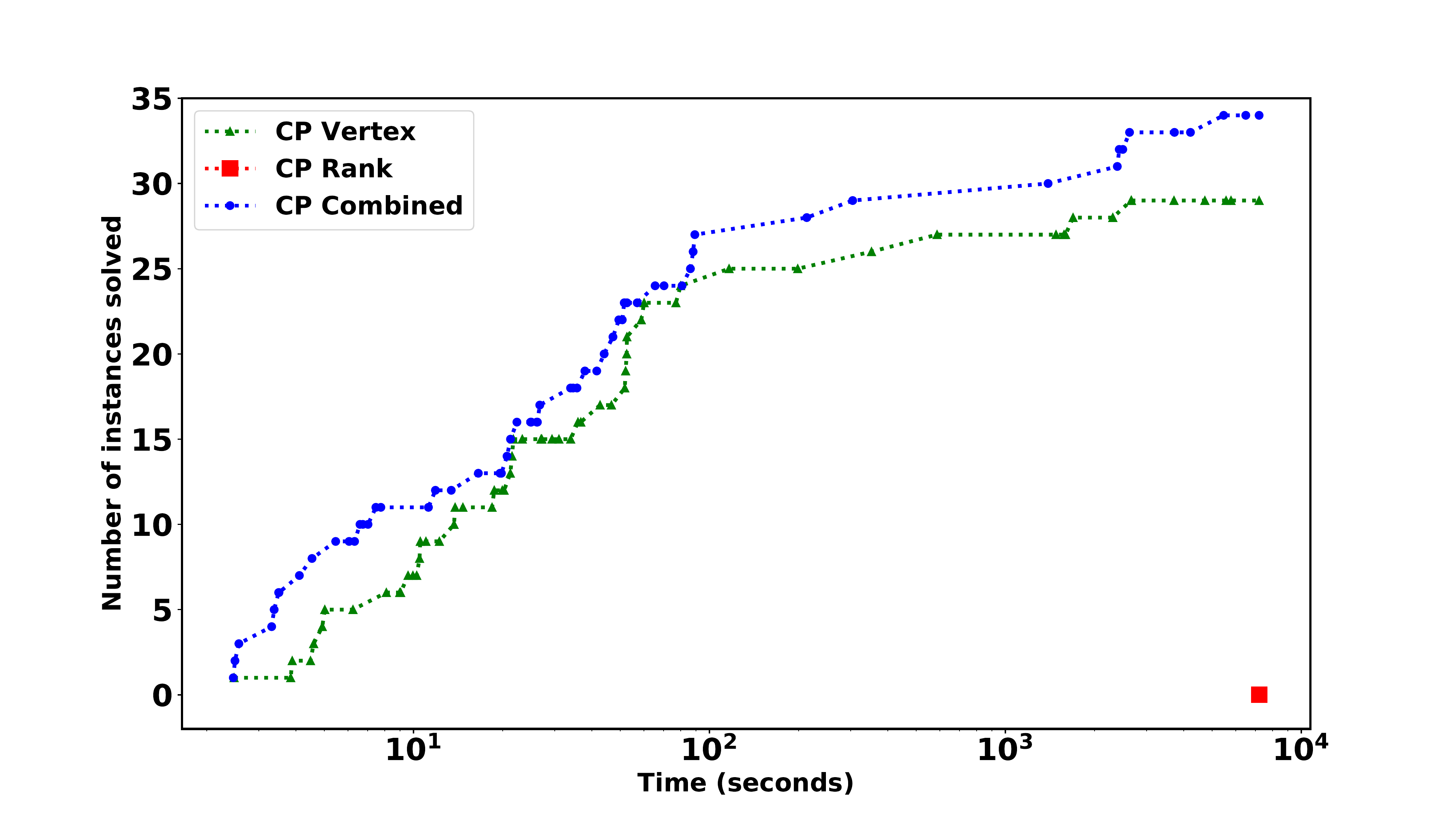}
		\caption{Medium-sized infeasible instances}
		\label{fig:LargeInfeas}
	\end{subfigure}
	\caption{Solution times of the {\crevOne constraint programming} models for medium-sized instances.}
	\label{fig:DMDGPLargeResults}
\end{figure}

The performance profiles for the solutions times of different models on medium-sized instances are given in Figure \ref{fig:DMDGPLargeResults} and the detailed results table can be found in Table 
	{\crevNet \ref{table:largeCP} 
%		of the Online Supplement.
	of	Appendix \ref{app:DMDGPApp}}. 
	We observe that $\rankDMDGP$ is outperformed by $\combDMDGP$ and $\vertexDMDGP$ in both the feasible and infeasible cases; it is unable to solve any infeasible instances. In the feasible case, $\vertexDMDGP$ solves $6$ more instances than $\combDMDGP$ within the time limit, however in the infeasible case,  $\combDMDGP$ solves $5$ more instances than $\vertexDMDGP $ within the time limit. Overall, $\vertexDMDGP$ is able to solve one more instance than  $\combDMDGP$ before the time limit is reached. } We note however, that for medium-sized instances {\crevOne $25$} of the instances were unsolved by any method. For these medium-sized instances, those with $D = 0.2$ or $D = 0.5$ were all solved in less than $2$ minutes. The  $D = 0.3$ and $D = 0.4$ instances are more difficult with only $13$ of $24$ instances with $D = 0.4$ solved and no instances with $D = 0.3$ and $n\geq 85$ solved. Thus there is still opportunity to improve the CP formulations.

{\crevOne Finally, {\crevNet we compare} the best performing CP formulations, $\combDMDGP$ and $\vertexDMDGP$, with $\IPVR$ on the pseudo-protein instances. As seen in Figure \ref{fig:ProtResults}, the results can be found in Tables
	{\crevNet \ref{table:prot30} to \ref{table:prot60} 
%		of the Online Supplement.
	of	Appendix \ref{app:DMDGPApp}}. 
	We observe that  $\vertexDMDGP$  and $\combDMDGP$ outperform $\IPVR$, which is only able to solve $39$ instances in total within the time limit.  $\combDMDGP$ has the best performance on these instances solving every instance in less than $32$ seconds. $\vertexDMDGP$ is unable to solve $9$ instances before the time limit is hit. We observe that of the instances $\vertexDMDGP$ is unable to solve, $8$ have $n \geq 40$ and they are those that have the highest density. Overall, we conclude that the CP formulations perform very well on pseudo-protein instances.}
\begin{figure}[h]
	\centering
	\hspace*{-0.7cm} 
		\includegraphics[scale=0.25]{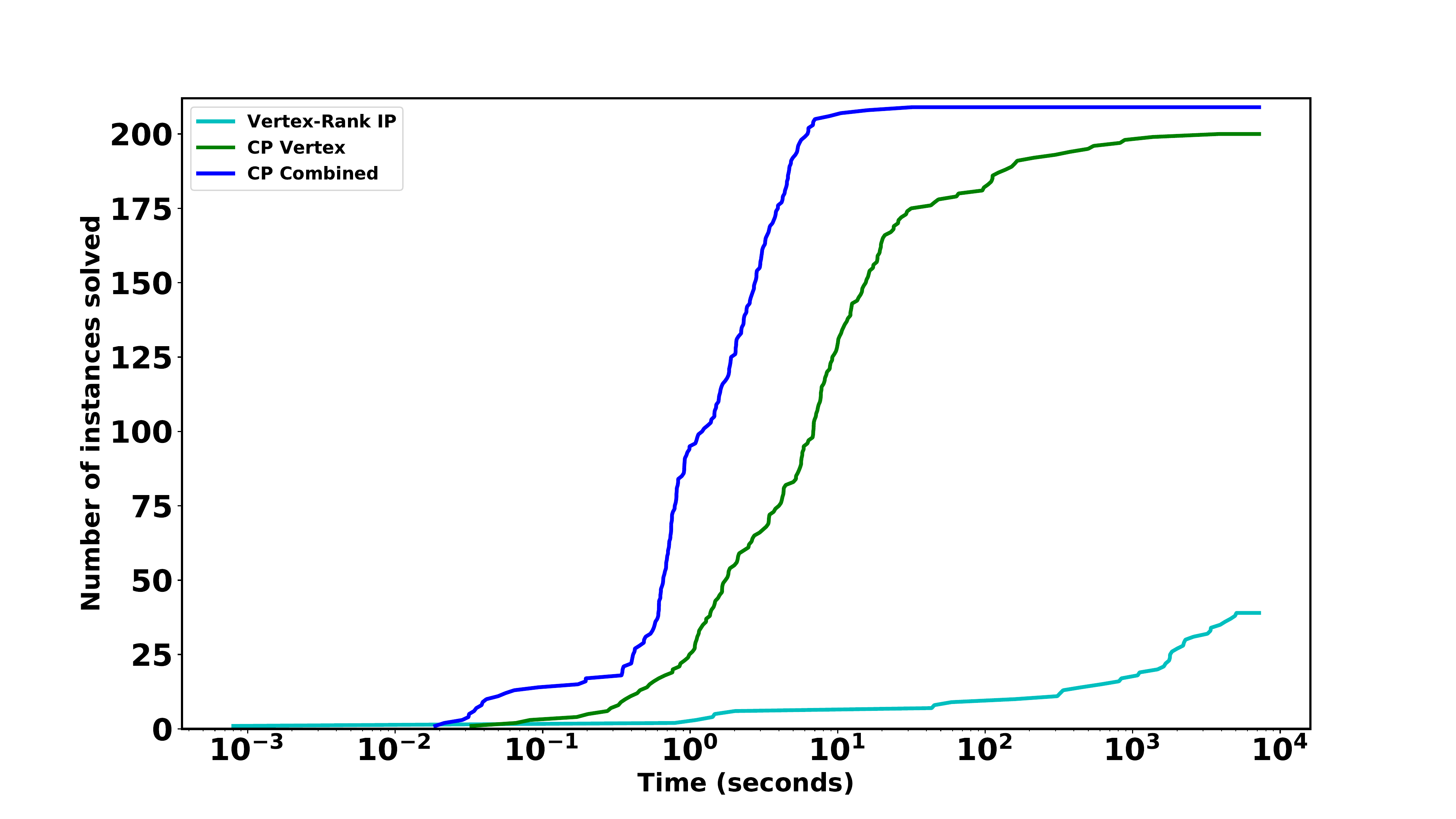}
	\caption{Solution times for pseudo-protein instances.}
	\label{fig:ProtResults}
\end{figure}

\subsection{\cb Preliminary Computational Study of  Structural Findings} \label{sec:struct}
{\cb In this section, we present a preliminary study of the structural findings of Section \ref{sec:enhance}. We implement them on $\combCP$ and test on the small random instances. {\crevNet In the worst case it takes 0.12 seconds to check for infeasibility and then apply domain reduction and symmetry breaking}.} For both the domain reduction and the symmetry breaking, we apply the infeasibility checks before solving. The infeasibility checks are able to prove four instances are infeasible without having to solve a mathematical program. %The implementation of the class of valid inequalities from Section \ref{sec:validineqs} is left as a future work. 

Solution times for $\combCP$ {\crev on small instances} with domain reduction {\crev only} and the symmetry breaking  {\crev only}, as well as both {\crev domain reduction and symmetry breaking together,} and with no additions are shown in Figure \ref{fig:DMDGPEnResults}. {\crev In total, for the set of $81$ small instances, Infeasibility Check \ref{inf:mindeg}  was able to solve two instances and Infeasibility Check \ref{inf:UB} was also able to solve two instances. Domain Reduction Rule \ref{red:smalldeg} was applied to $12$ instances. Finally, Symmetry Breaking Condition \ref{sym:ex_Ss} was applied to $33$ instances, Symmetry Breaking Condition \ref{sym:ex_clique} was applied to $32$ instances, and Symmetry Breaking Condition \ref{sym:arb} was applied to $37$ instances.} {\crev None of the instances could be deemed infeasible using Infeasibility Checks \ref{inf:minedge} and \ref{inf:LB}. We were also unable to apply Domain Reduction Rule \ref{red:smalldegneighb}, and could not use Symmetry Breaking Conditions \ref{sym:degk},  \ref{sym:ss}, \ref{sym:clique} for these instances.}  

We present the results in Table 
{\crevNet \ref{table:enhance} 
%	of the Online Supplement. %
		of Appendix \ref{app:DMDGPApp}}. 
We focus on small instances since we were unable to apply any of {\cb the findings} other than arbitrary symmetry breaking to the medium-sized instances. We also see that the solution time with the addition of some combination of the {\crevNet structural strategies} is decreased for $63$ of the $81$ instances. We believe this is because as $n$ increases,  an instance is less likely to have degree  less than $2\K$, since $\K$ is small with respect to $n$ and it is less likely that two vertices will have the same neighbourhoods. We observe that for feasible instances, {\crevOne symmetry breaking is slightly detrimental to the performance of $\combCP$ and that adding only domain reduction gives very similar results to no additions at all.  However, for infeasible instances, we observe that domain reduction alone improves upon $\combDMDGP$, and has similar performance to when both symmetry breaking and domain reduction are added to the model.}

\begin{figure}[h]
\hspace*{-0.7cm} 
	\begin{subfigure}{0.5\textwidth}
		\includegraphics[scale=0.23]{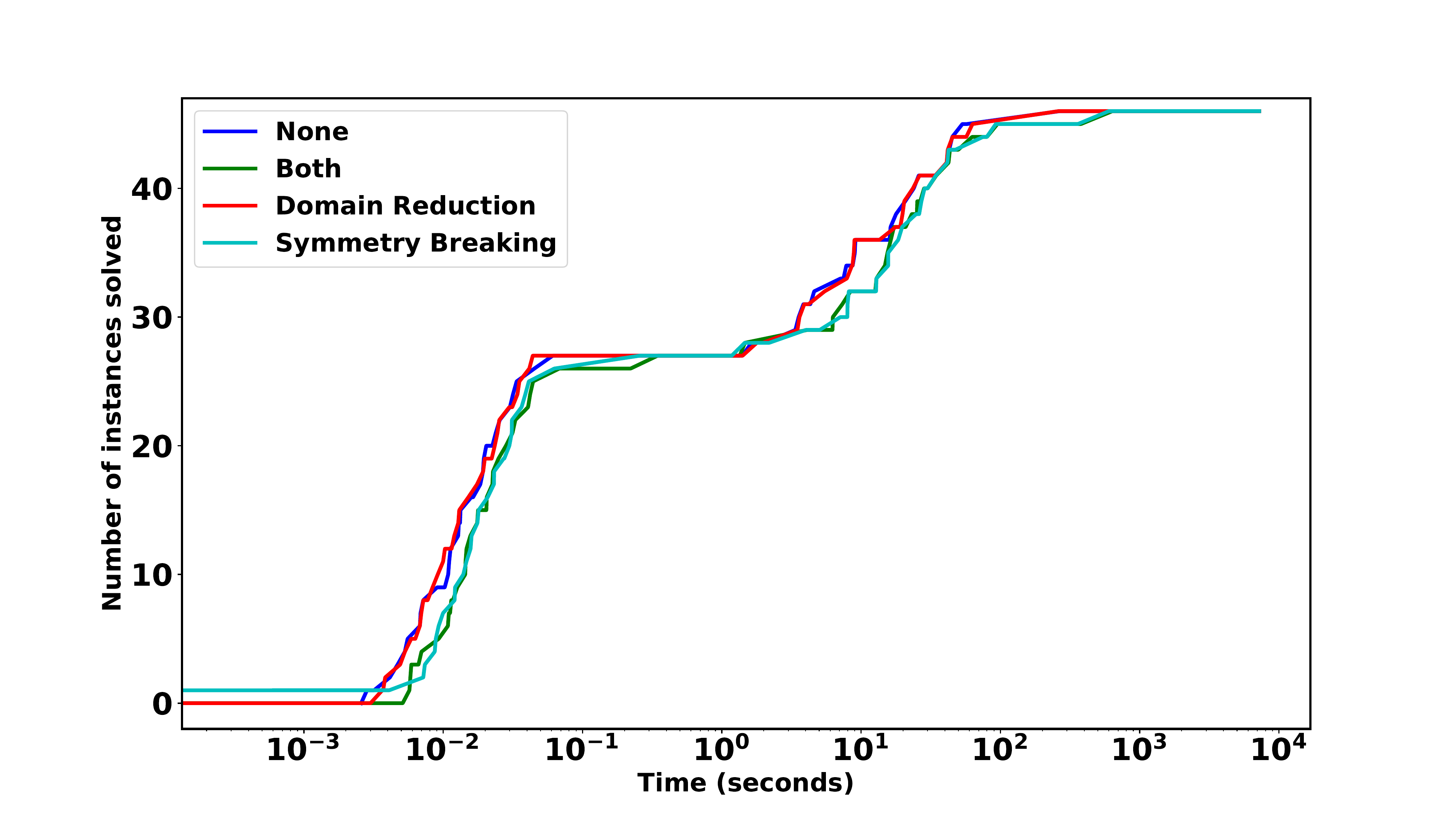}
		\caption{Small feasible instances}
		\label{fig:combinedDMDGPFeas}
	\end{subfigure}
\hspace*{0.1cm} 
	\begin{subfigure}{0.45\textwidth}
		\includegraphics[scale=0.23]{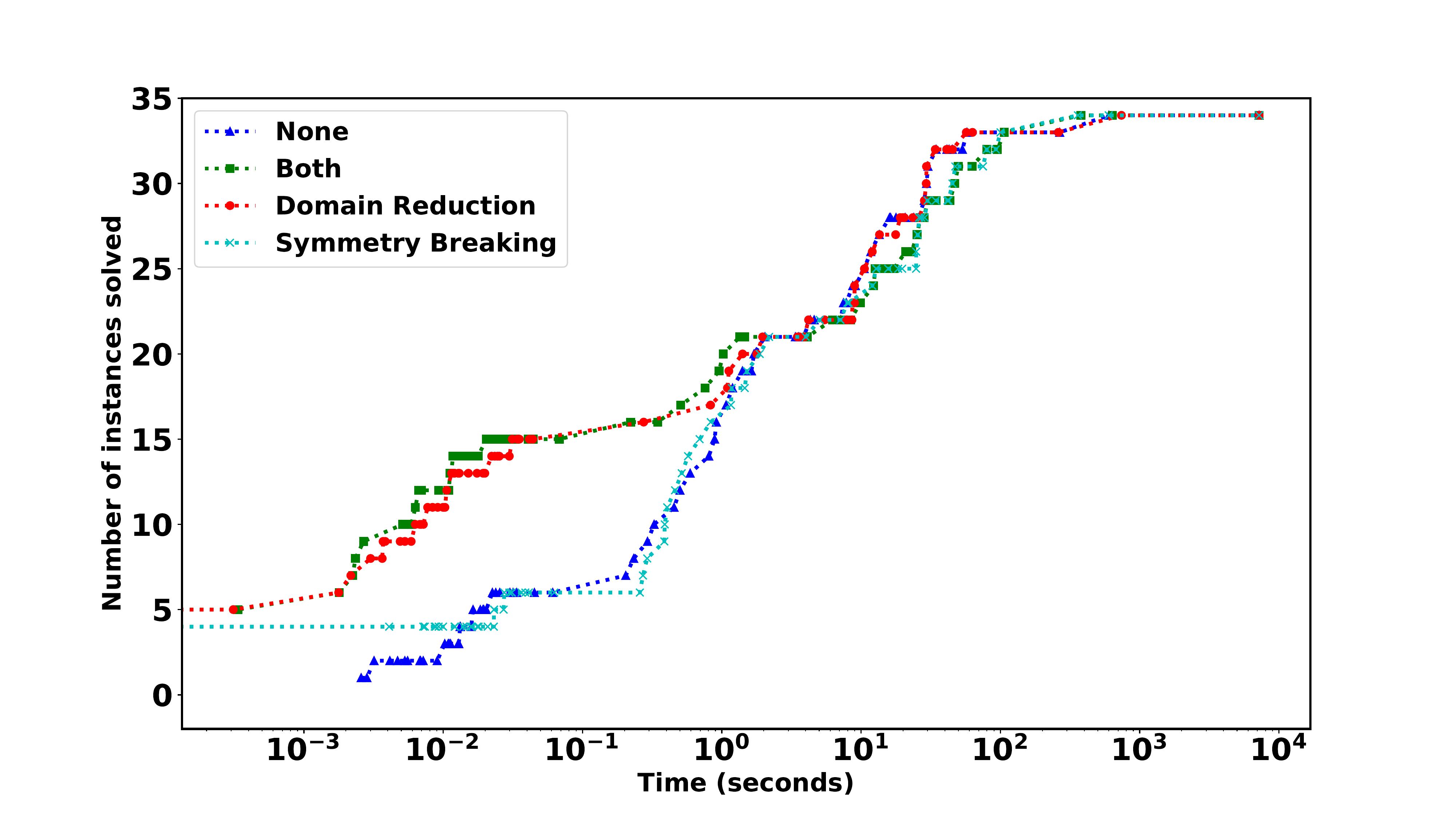}
		\caption{Small infeasible instances}
		\label{fig:combinedDMDGPInfeas}
	\end{subfigure}
	\caption{Solution times of the alternatives for structural findings.}
	\label{fig:DMDGPEnResults}
\end{figure}

For the majority of instances there exists an addition of one or more structural finding to the model that improves the solution time as compared to $\combDMDGP$ on its own. Thus we conclude that the structural findings are beneficial for this data set. We also see that as the instance size increases we are unable to find these key structures in the graph. This motivates the future work of incorporating these structural insights into the CP search tree for further propagation opportunities and apply them at every node of the search tree to some subset of the order.

\subsection{\cb Preliminary Computational Study of Valid Inequalities} \label{sec:vi_comp}

To provide insight into the class of valid inequalities using stable sets, \eqref{eq:VIabs} and \eqref{eq:VIMinD}, presented in Section \ref{sec:validineqs}, we study the impact of adding all possible valid inequalities to both $\rankDMDGP$ and $\combDMDGP$ before solving. We use the Networkx Python package to enumerate all maximal stable sets for each small instance; {\crevNet the total time to enumerate these sets is $2.55$ seconds}. We then add the valid inequalities  \eqref{eq:VIabs} or \eqref{eq:VIMinD} {\crevNet to the model as constraints} using these stable sets and compare against the formulations without the valid inequalities. {\cb We remark that this is by no means a complete study of the impact of the valid inequalities but is meant to provide insight that even with a naive implementation there are cases where the valid inequalities are useful.} The complete results table can be found in Table 
{\crevNet \ref{table:VI}  of 
%	of the Online Supplement.
	Appendix \ref{app:DMDGPApp}}

Figure \ref{fig:VIFeas} shows the results for small feasible instances. We observe that adding the valid inequalities was only slightly detrimental for both  $\rankDMDGP$ and $\combDMDGP$ and that the inequalities \eqref{eq:VIabs} have better performance than the inequalities \eqref{eq:VIMinD}. In the infeasible case, Figure \ref{fig:VIInfeas}, both forms of the valid inequalities increase the number of instances solved by  $\rankDMDGP$ within the time limit and the speed at which they were solved. For  $\combDMDGP$ on {\crev the smallest} infeasible instances, the formulations with valid inequalities \eqref{eq:VIabs}  and  \eqref{eq:VIMinD}  dominate the formulation without. In fact, the formulations with valid inequalities are able to solve some of these instances with $0$ choice points, which is notable. As the size of the instances grows,  $\combDMDGP$ without valid inequalities has better performance. This is most likely as a result of the larger instances having more stable sets and the subsequent models having a large number of constraints as a result.
\begin{figure}[h]
	\hspace*{-0.7cm} 
	\begin{subfigure}{0.5\textwidth}
		\includegraphics[scale=0.23]{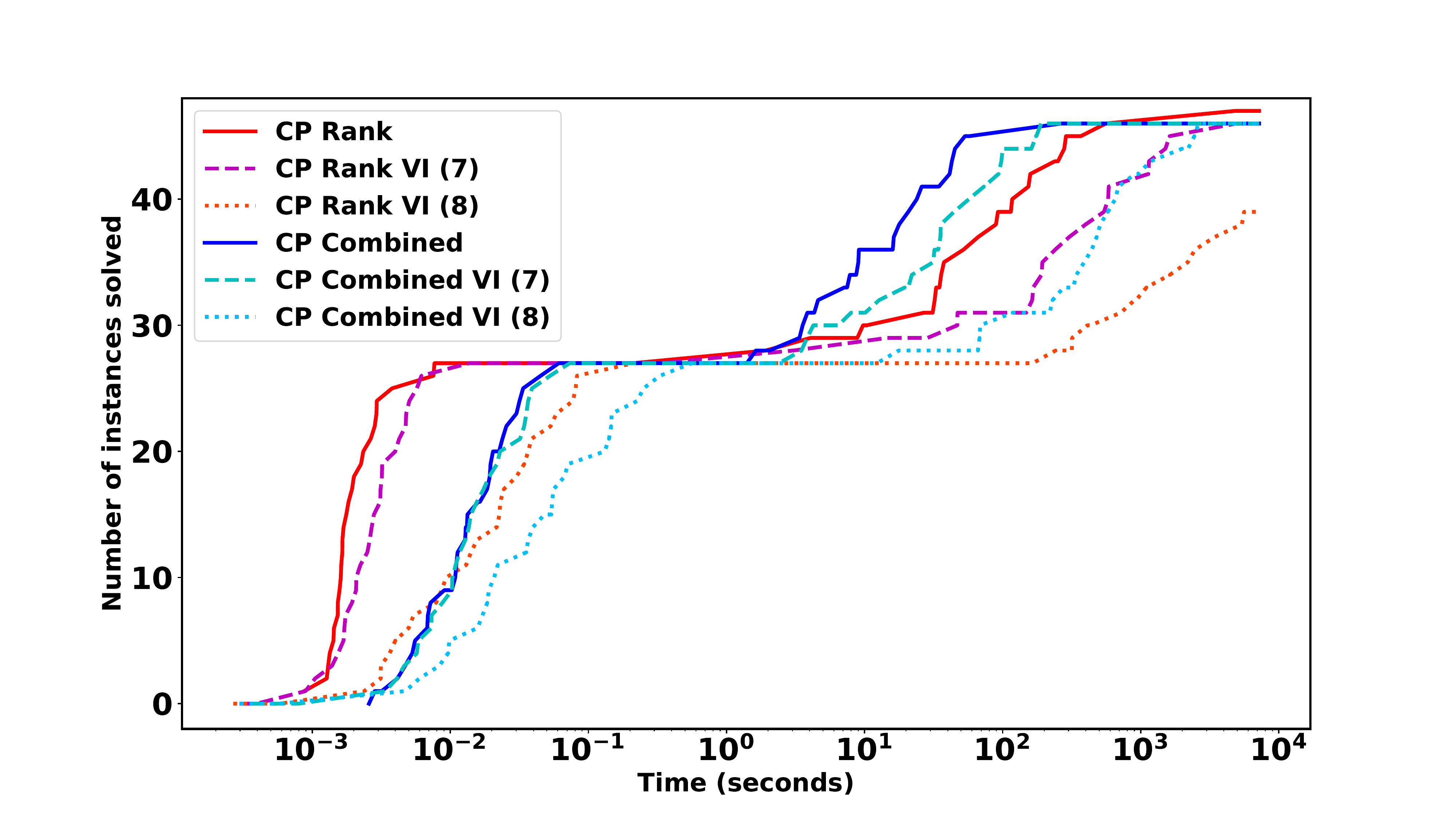}
		\caption{Small feasible instances}
		\label{fig:VIFeas}
	\end{subfigure}
	\hspace*{0.1cm} 
	\begin{subfigure}{0.45\textwidth}
		\includegraphics[scale=0.23]{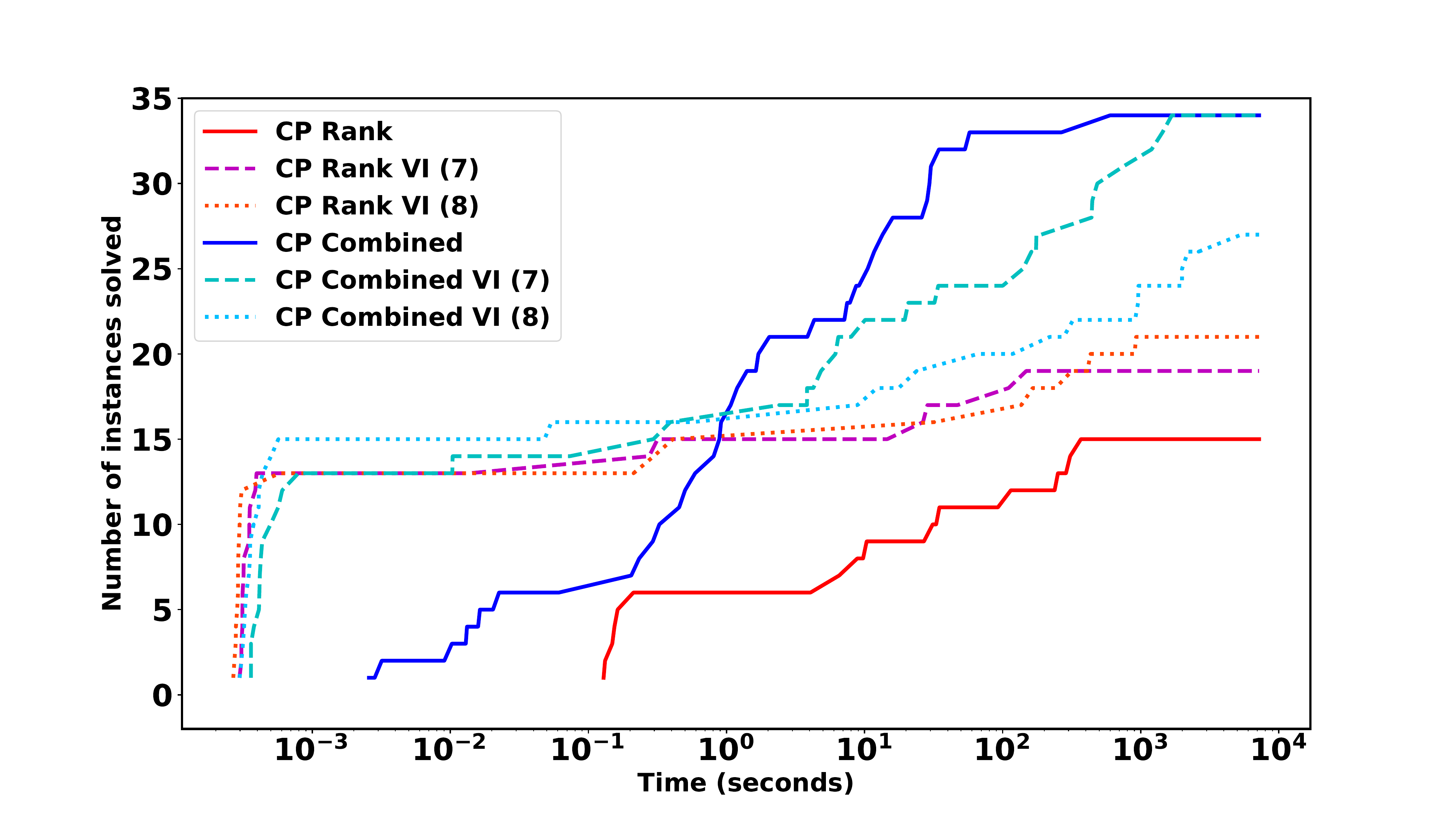}
		\caption{Small infeasible instances}
		\label{fig:VIInfeas}
	\end{subfigure}
	\caption{Solution times of the alternatives for valid inequalities {\crev (VI)}.}
	\label{fig:DMDGPVIResults}
\end{figure}

We conclude that in this preliminary implementation the valid inequalities are useful in small, low density instances and to improve $\rankDMDGP$ on infeasible instances. Future work on this topic includes further analysis to select a small enough subset of the stable sets that the number of constraints will not be too large and that will have a large impact on solution times. We remark however that the valid inequalities show promise given a more suitable implementation, such as a specialized propagator for the CP formulations.
\section{Conclusion} \label{sec:conclusion}

We propose the first CP formulations for the CTOP and compare them against two existing IP formulations in the literature. {\cb Applying ideas from both the CP and IP literature, and by studying the structure of DMDGP orders,} we introduce three classes of additions to formulations for CTOP, namely infeasibility checks, domain reduction, and symmetry breaking. We also provide the first class of valid inequalities for CTOP.

Our computational results show our models outperform the state-of-the-art IP formulations. They also indicate that the {\cb additions to the models based on structural findings} are useful for infeasible instances, but may negatively impact the amount of time it takes to solve feasible instances. {\crevOne Nonetheless, they give important insight into the structure of CTOP and show promise for improvement and extension.} 
% Acknowledgments here
\paragraph*{Acknowledgements}
{\crevOne We are grateful to Leo Liberti for introducing us to the area of distance geometry and this topic in particular. We would also like to thank J\'{e}r\'{e}my Omer for kindly providing the pseudo-protein instances similar to those used in \cite{Omer2017}. }

\bibliographystyle{abbrv-networks}
{\small \bibliography{bibliography}}

\newpage
\appendix
\section{Details of the Existing IP Models} 
\label{app:IPapp}
Prior to this work, Cassioli et al.\ \cite{cassioli2015} present two IP formulations for CTOP. 
\subsection{The vertex-rank IP}
\label{app:VRIP}
Let $\vrvar_{\vertind\rankind}$ be a binary variable, which takes value $1$ if a vertex $\vertind \in \vertexset$ is receives rank $\rankind \in [n-1]$ in the order, and $0$ otherwise. Since CTOP is a satisfiability problem, we are simply looking for a feasible order; there is no objective. The so-called \emph{vertex-rank IP} formulation is as follows:
\bsubeq
\label{form:IPvr}
\begin{alignat}{3}
\IPVR: \ & \sum_{\rankind \in  [n-1] } \vrvar_{\vertind\rankind} = 1  &  & \forall \  \vertind \in \vertexset \label{IPvr:vertex_unique} \\
&\sum_{\vertind \in \vertexset}  \vrvar_{\vertind\rankind} = 1 & & \forall \ \rankind \in  [n-1] \label{IPvr:rank_unique} \\
& \sum_{u \in \neighb(\vertind)} \sum_{j\in [r-1]} \vrvar_{uj} \geq \rankind\vrvar_{\vertind\rankind} & &\forall \  \vertind \in \vertexset, \rankind \in [1,\K-1] \label{IPvr:init_clique}\\
& \sum_{u \in \neighb(\vertind)} \sum_{j \in [\rankind-\K,\rankind-1]} \vrvar_{uj} \geq \K \vrvar_{\vertind\rankind} \quad & &  \forall \ \vertind \in \vertexset, \rankind \in [\K, n-1] \label{IPvr:dmdgp_pred} \\
& \vrvar_{\vertind \rankind} \in \{0,1\} && \forall \ \vertind \in \vertexset, \rankind \in [n-1] \label{IPvr:domain}
\end{alignat}
\esubeq

\noindent Constraints \eqref{IPvr:vertex_unique} and \eqref{IPvr:rank_unique} enforce a one-to-one assignment between the vertices and the ranks, so that each vertex appears exactly once in the order and that each rank gets exactly one vertex. Constraints \eqref{IPvr:init_clique} enforce that there must be an initial clique of size $\K$, i.e., that the vertices in positions $[\K-1]$ are all adjacent to their predecessors. Constraints \eqref{IPvr:dmdgp_pred} enforce each vertex with a rank in $[\K, n-1]$ has at least $\K $ contiguous predecessors. Finally, constraints \eqref{IPvr:domain} ensure the binary domain of the vertex-rank variables.

%Next we provide the proof of Proposition \ref{prop:LPrelax} showing that  the LP relaxation of $\IPVR$ is always feasible.  
\medskip 
{\bf Proof of Proposition \ref{prop:LPrelax}.}
	 We claim that setting $\vrvar_{\vertind \rankind} = \frac{1}{n}$ for all $\vertind \in \vertexset, \rankind \in [n-1]$ always yields a feasible solution to the LP relaxation of $\IPVR$. We show the proposed solution satisfies all LP constraints. For constraints \eqref{IPvr:vertex_unique},
	%\[ \forall  \ \vertind \in \vertexset\displaystyle \sum_{r \in [n-1]} x_{vr} = 1 \]
	fixing $\vertind \in \vertexset$ gives
	\[
	\displaystyle\sum_{r \in [n-1]} \frac{1}{n} = n  \cdot \frac{1}{n} = 1.
	\]
Similarly, for constraints \eqref{IPvr:rank_unique}, fixing $r \in [n-1]$ gives
	%\[ \forall \ r \in [n-1]\displaystyle \sum_{\vertind \in \vertexset} x_{vi} = 1 \]
	\[ \displaystyle \sum_{\vertind \in \vertexset} \frac{1}{n}  = |\vertexset| \cdot  \frac{1}{n}  = n\cdot  \frac{1}{n}  = 1. \]
	For constraints \eqref{IPvr:init_clique}, for any $\vertind \in \vertexset$ and $\rankind \in [1,\K-1]$ we have
	%\[ \sum_{u \in \neighb(\vertind)} \sum_{j\in [r-1]} \vrvar_{uj} \geq \rankind\vrvar_{\vertind\rankind} \]
	%Thus,
	\[
	\sum_{u \in \neighb(\vertind)} \sum_{j\in [r-1]}  \frac{1}{n}\geq \sum_{j \in [r-1]}\frac{1}{n} = r  \frac{1}{n}
	\]
	where the inequality follows from $|\neighb(\vertind)| \geq 1$ since $\graph$ is connected. 
	Finally, for constraints \eqref{IPvr:dmdgp_pred}, for any $\vertind \in \vertexset$ and $\rankind \in [\K, n-1]$ we have
	\[
	 \sum_{u \in \neighb(\vertind)} \sum_{j \in [\rankind-\K,\rankind-1]} \frac{1}{n}  \geq  \sum_{j \in [\rankind-\K,\rankind-1]} \frac{1}{n} = \K \cdot \frac{1}{n} 
	\]
	again, due to $|\neighb(\vertind)| \geq 1$ since G is connected.
	
	Thus $\vrvar_{\vertind \rankind} = \frac{1}{n}$ satisfies all constraints and the LP relaxation is feasible. 
 \hfill $\Box$

%\begin{prop}
%	LP relaxation and S-cuts are not equal to IP formulation
%	
%	\textcolor{blue}  {Is this true? How can we prove it? Is it helpful?}
%\end{prop}
%\begin{proof}
%\end{proof}

\subsection{The clique digraph IP and its relaxation}
\label{app:CDIP}
As stated in Key Property \ref{prop:key}, a DMDGP order is a series of overlapping induced $(\K+1)$-cliques in $\graph$, which cover all the vertices and share $\K$ vertices between adjacent pairs. Define a clique digraph $\cliquedigraph = (\orderedcliques, \arcset)$, where $\orderedcliques$ is the index set of all \emph{ordered} cliques $\{\orderclique_j\}_{j \in \orderedcliques}$ of size $\K+1$  in the input graph $\graph$, where $\orderclique_j =(v_1^j, v_2^j, \hdots, v_{\K+1}^j)$, i.e., represented simply by its ordered vertices. There is an arc $(\orderclique_i,\orderclique_j) \in \arcset$ between $\orderclique_i,\orderclique_j \in \orderedcliques$ if $v_{\K+1}^i = v_{\K}^j$, i.e., if the two cliques overlap by $\K$ vertices and differ only by the first and the last vertex respectively. For instance, in the example given in Figure \ref{fig:DMDGPexample}, there will be an arc in $\arcset$ between the vertices corresponding to the ordered 3-cliques $(\vertind_4,\vertind_2,\vertind_3)$ and $(\vertind_2,\vertind_3,\vertind_1)$. Let $\ell_i$ be the last vertex of a clique $\orderclique_i \in \orderedcliques$. In this setting a DMDGP order is described by a path (of cliques) $P = (o_1, o_2, \hdots, o_{n-\K})$ in $\cliquedigraph$ where $ \{v \in \orderclique_{c_1}\} \cup \{\ell_{i} : i \in [2, n-\K] \} = \vertexset$. That is, the initial clique and the last vertices of all other cliques cover $\vertexset$. For instance, the DMDGP order given in Figure \ref{fig:DMDGPcliques} is described by the path of cliques $o_1 = (\vertind_4,\vertind_2,\vertind_3), o_2 = (\vertind_2,\vertind_3,\vertind_1), o_3 = (\vertind_3,\vertind_1,\vertind_5)$ and $o_4 = (\vertind_1,\vertind_5,\vertind_0)$. 

Define binary variables $\arcvar_{ij} = 1$ if the arc $(i,j) \in \arcset$ is selected in the path solution $P$, $0$ otherwise. Let the binary variable $\initcliquevar_j = 1$ if $j \in \orderedcliques$ is the first clique in $P$ and $\lastcliquevar_j= 1$ if $j \in \orderedcliques$ is the last clique in $P$. Define binary precedence variables $\predvar_{uv} = 1$ if $u \in \vertexset$ precedes $v \in \vertexset$ in the DMDGP order. Then the \emph{clique digraph IP} formulation is as follows: 
\bsubeq\label{form:cliqueDigraph}
\begin{alignat}{2}
\IPCD: \min \ & \sum_{(i,j) \in \arcset} \arcvar_{ij} \label{IPcd:obj} \\
\text{s.t.} \ & \sum_{j \in \orderedcliques} \initcliquevar_j=1 \label{IPcd:firstClique} \\
& \sum_{j \in \orderedcliques} \lastcliquevar_j =1 \label{IPcd:lastClique}  \\
& \initcliquevar_i + \sum_{\substack{j \in \orderedcliques \\ (j,i) \in \arcset}} \arcvar_{ji} = \lastcliquevar_i + \sum_{\substack{j \in \orderedcliques \\ (i,j) \in \arcset}} \arcvar_{ij} \quad &&  \forall\  i \in \orderedcliques \label{IPcd:flowBalance}\\
&\sum_{\substack{j \in \orderedcliques: \\ (i,j) \in \arcset}} \arcvar_{ij} \leq 1 &&  \forall\  i \in \orderedcliques \label{IPcd:successor}\\
& \sum_{\substack{j \in \orderedcliques: \\ v \in \orderclique_j}}\initcliquevar_j  +\sum_{\substack{(i,j) \in \arcset: \\ \orderclique_j \setminus \orderclique_i = \{v\}}} \arcvar_{ij} = 1 &&  \forall\  v \in \vertexset \label{IPcd:coverPath} \\
& \predvar_{uv} + \predvar_{vu} = 1 && \forall \ v, u \in \vertexset \text{ s.t. } v \neq u \label{IPcd:LOpairs}\\
&\predvar_{uv} + \predvar_{vw} + \predvar_{wu} \leq 2 && \forall \ v, u, w \in \vertexset \text{ s.t. } v \neq u \neq w\label{IPcd:LOtriplets}\\
& \predvar_{v_{k}^iv_{k+1}^i} \geq \arcvar_{ij}&& \forall \ (i,j) \in \arcset, k \in  [1,\K] \label{IPcd:cliqueOrder1} \\
&  \predvar_{v_{K+1}^iv} \geq \arcvar_{ij}&& \forall \ (i,j) \in \arcset,  v = \orderclique_j \setminus \orderclique_i%v \in \orderclique_j \setminus \orderclique_i \text{ s.t. } |\orderclique_j \setminus \orderclique_i | = 1 
\label{IPcd:cliqueOrder2}\\
& \predvar_{v^i_k v^i_{k+1}} \geq \initcliquevar_i + \lastcliquevar_i  && \forall \ i \in \orderedcliques, k\in [1,\K] \label{IPcd:cliqueOrder3}\\
&\initcliquevar^i \in \{ 0,1\}&& \forall \ i \in \orderedcliques \label{IPcd:domInitClique} \\
&\lastcliquevar^i \in \{ 0,1\} && \forall \ i \in \orderedcliques \label{IPcd:domLastClique} \\
& \arcvar_{ij} \in \{ 0,1\} && \forall \ (i,j) \in \arcset \label{IPcd:domArc} \\
& \predvar_{uv} \in \{ 0,1\} && \forall \ u, v \in \vertexset \label{IPcd:domPred}
\end{alignat} 
\esubeq
Objective \eqref{IPcd:obj} imposes that we will select the minimum number of arcs required to form the path $P$. Constraints \eqref{IPcd:firstClique} and \eqref{IPcd:lastClique} ensure there is exactly one initial clique and one last clique selected. Constraints \eqref{IPcd:flowBalance} ensure that flow balance holds in the path $P$ except at the first and last path vertices which have one arc out and one arc in respectively. These flow balance constraints also ensure a correct predecessor relationship between the cliques in $P$. Constraints \eqref{IPcd:successor} ensure each clique has at most one successor, one if it is in the path and not the last clique and none otherwise. Constraints \eqref{IPcd:coverPath} ensure that the cliques selected cover all the vertices in $\vertexset$. Constraints \eqref{IPcd:LOpairs} and \eqref{IPcd:LOtriplets} impose a linear order among vertex pairs and triplets. Constraints \eqref{IPcd:cliqueOrder1}, \eqref{IPcd:cliqueOrder2}, and \eqref{IPcd:cliqueOrder3} \footnote{In \cite{cassioli2015}, the variables for predecessors are given as $w_{uv}$, only in these constraints. We believe this is a typo that the variables are in fact $\predvar_{uv}$.} ensure that each clique is ordered.
 Constraints \eqref{IPcd:cliqueOrder1} impose that $v^i_k$ precedes vertex $v^i_{k+1}$ if ordered clique $i$ has an outgoing arc in the path solution $P$. Constraints \eqref{IPcd:cliqueOrder2} ensure that if arc $(i,j) \in \arcset$ is selected in $P$, the vertex of $j$ not in $i$, $v$, is preceded by all other vertices of $j$ in $i$, which have been ordered by \eqref{IPcd:cliqueOrder1}. Constraints \eqref{IPcd:cliqueOrder3} are similar to \eqref{IPcd:cliqueOrder1}, except they order the vertices of the first and last clique. Finally, constraints \eqref{IPcd:domInitClique}, \eqref{IPcd:domLastClique}, \eqref{IPcd:domArc}, and \eqref{IPcd:domPred} enforce the binary domains of all variables.

$\IPCD$ is disadvantaged by the potential number of vertices in the clique digraph $D$, as the cardinality of $\orderedcliques$ can be quite large even for relatively sparse graphs. To reduce the number of variables in $\IPCD$, Cassioli et al.\ \cite{cassioli2015} present a relaxation of the clique digraph formulation which considers unordered cliques. The idea is to relax the ordering constraints in the formulation and to solve this relaxation as a first check for the existence of a DMDGP order. In this case, the worst case number of vertices in $\cliquedigraph$ can be reduced by a factor of $(\K+1)!$. Let $\orderedcliques$ now denote the set of \emph{unordered} cliques of size $\K+1$ in $\graph$. Let binary variable $\unorderedindvar_j= 1$ if the unordered clique $j \in \orderedcliques$ is used in $P$. The unordered clique \emph{IP relaxation} of $\IPCD$ is as follows:
\bsubeq
\label{form:unorderedRelax}
\begin{alignat}{2}
\Unorderedrelax: \min \ & \sum_{(i,j) \in \arcset} \arcvar_{ij} \label{IPrelax:obj} \\
\text{s.t.} \ & \eqref{IPcd:firstClique} - \eqref{IPcd:LOtriplets} \label{IPrelax:existing}\\
& \predvar_{uv} \geq \arcvar_{ij} && \forall \ (i,j) \in \arcset, u \in \orderclique_i, v = \orderclique_j \setminus \orderclique_i %v \in \orderclique_j \setminus \orderclique_i \text{ s.t. } |\orderclique_j \setminus \orderclique_i | = 1
\label{IPrelax:cliqueUnorder}\\
& \unorderedindvar_j \geq \arcvar_{ij} && \forall \ (i,j) \in \arcset \label{IPrelax:ind1} \\
& \unorderedindvar_i \geq \initcliquevar_i && \forall \ i \in \orderedcliques \label{IPrelax:ind2} \\
& \sum_{\substack{i \in \orderedcliques : \\ v \in \orderclique_i}} \unorderedindvar_i \leq \K + 1 \quad && \forall \  v \in \vertexset  \label{IPrelax:relax} \\
& \eqref{IPcd:domInitClique} - \eqref{IPcd:domArc}
 \label{IPrelax:dom1} \\
&z_i \in \{0,1\} && \forall \ i \in \orderedcliques  \label{IPrelax:dom2}
\end{alignat}
\esubeq
\eqref{IPrelax:obj} and\eqref{IPrelax:existing} are the same as in $\IPCD$.  However, we have relaxed the clique ordering constraints, \eqref{IPcd:cliqueOrder1}-\eqref{IPcd:cliqueOrder3}, so now  constraints \eqref{IPrelax:cliqueUnorder} ensure we have that if arc $(i,j) \in \arcset$ is selected in $P$, all $u \in \orderclique_i$ precede vertex $v$, the only vertex of $j$ not in $i$. Constraints \eqref{IPrelax:ind1} ensure we have correctly linked the arc variables and clique variables to the indicator $\unorderedindvar_j$, so that it is $1$ if a clique is part of an arc selected in the solution path $P$, while the indicator for the first clique is activated through constraints \eqref{IPrelax:ind2}. The constraints \eqref{IPrelax:relax} impose that each vertex appears in at most $\K+1$ cliques. These constraints are another relaxation, since to make it exact we would need to enforce that all vertices except the first and last $\K$ appear in $\K+1$ cliques; however this will require  many more variables to express. Finally constraints \eqref{IPrelax:dom1} and \eqref{IPrelax:dom2} enforce the variable domains.

When a solution to $\Unorderedrelax$ is found, it must be verified as this solution does not necessarily yield a DMDGP order. The verification is a simple check to ensure the $\predvar$ solution forms a DMDGP order. The strength in this formulation is that if $\Unorderedrelax$ is infeasible, there is no DMDGP order for the instance.
\section{Selected Proofs}
\label{Proofapp}

\subsection{Proof of Proposition \ref{edge_min}} 
\label{proof_edgemin}
\begin{proof}
	Minimally, the vertex in position $0$ must be adjacent to $\K$ vertices since it is in a $(\K+1)$-clique, similarly for the vertex in position $n-1$. The vertex in position $1$ is in two $(\K+1)$-cliques since it is in the initial clique but also is an adjacent predecessor of the vertex in position $\K+1$; similarly for the vertex in position $n-2$. We extend this logic to all vertices in the order, noting that the centre $|\vertexset| - 2\K$ vertices in the order, i.e., the ones in positions $[\K, n-\K-1]$, must all be in $2\K$ cliques, as seen in Figure \ref{fig:mindegs}. 
	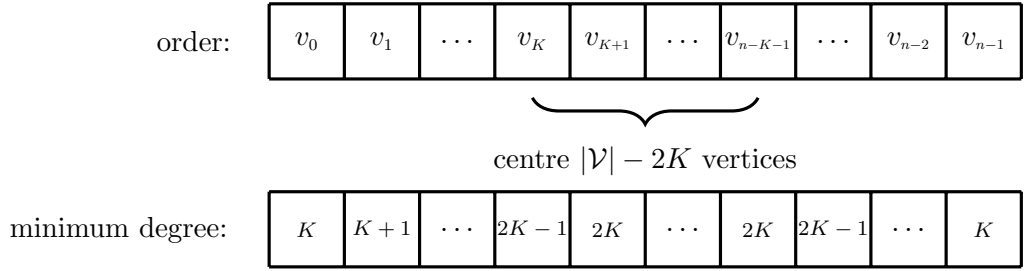
\begin{figure}[h]	
		\begin{tikzpicture}[main_node/.style={circle,fill=white!80,draw,inner sep=0pt, minimum size=18pt},
		line width=1.2pt]
		\draw (0,0) grid (10,1);
		\node[black] at (-1,0.5) {order: };
		\node[black] (v0) at (0.5,0.5) {$v_{\scaleto{0}{4pt}}$};
		\node[black] (v1) at (1.5,0.5) {$v_{\scaleto{1}{4pt}}$};
		\node[black] (dots1) at (2.5,0.5) {$\hdots$};
		\node[black] (vK) at (3.5,0.5) {$v_{\scaleto{\K}{4pt}}$};
		\node[black] (vK1) at (4.5,0.5) {$v_{\scaleto{\K+1}{4pt}}$};
		\node[black] (dots2) at (5.5,0.5) {$\hdots$};
		\node[black] (vnK1) at (6.5,0.5) {$v_{\scaleto{n-\K-1}{4pt}}$};
		\node[black] (dots3) at (7.5,0.5) {$\hdots$};
		\node[black] (vn2) at (8.5,0.5) {$v_{\scaleto{n-2}{4pt}}$};
		\node[black] (vn1) at (9.5,0.5) {$v_{\scaleto{n-1}{4pt}}$};
		\draw [decorate,decoration={brace,amplitude=10pt},xshift=10cm, rotate=90] (-0.25,3.5) -- (-0.25,6.5) node [black,below, midway, yshift=-0.5cm] {centre $|\vertexset| - 2\K$ vertices};
		\node[black] at (-2 ,-2) {minimum degree: };
		\draw[yshift=-0.5cm] (0,-2) grid (10,-1);
		\node[black] at (0.5,-2) {\scriptsize$\K$};
		\node[black] at (1.5,-2) {\scriptsize$\K+1$};
		\node[black]  at (2.5,-2) {$\hdots$};
		\node[black] at (3.5,-2) {\scriptsize$2\K-1$};
		\node[black] at (4.5,-2) {\scriptsize$2\K$};
		\node[black] at (5.5,-2) {$\hdots$};
		\node[black]  at (6.5,-2) {\scriptsize$2\K$};
		\node[black] at (7.5,-2) {\scriptsize$2\K-1$};
		\node[black]  at (8.5,-2) {$\hdots$};
		\node[black]  at (9.5,-2) {\scriptsize$\K$};
		\end{tikzpicture}
		\caption{Minimum degree requirement for ranks in a DMDGP order. }\label{fig:mindegs}
	\end{figure}
	This  minimal case analysis gives a lower bound on the number of edges in a DMDGP order. For each vertex from position $0$ to position $n-1$ we sum over the minimum degree in $\graph$ to have a DMDGP order. Note that we divide by two to avoid double counting since the input graph is undirected.
	\bsubeq 
	\begin{alignat*}{2}
	&\frac{1}{2} (\K + (\K +1) + (\K+2) + \cdots +2\K +\cdots +2\K +\cdots +(\K+2) +  (\K +1) + \K) \\
	& = \frac{ (|\vertexset|-2\K)2\K + 2\sum_{i = \K}^{2\K-1} i }{2}\\
	&=  \K(|\vertexset|-2\K) +\sum_{i = \K}^{2\K-1} i \\
	&= \K(|\vertexset|-2\K) + \frac{1}{2}\K(3\K-1) \\
	&= \left (|\vertexset| - \frac{1}{2} \right )\K - \frac{1}{2} \K^2
	\end{alignat*}
	\esubeq
	Thus the minimum number of edges for a DMDGP order is $\left (|\vertexset| - \frac{1}{2} \right )\K - \frac{1}{2} \K^2$.
\end{proof}

\subsection{Proof of Proposition \ref{prop_incompare}} 
\label{proof_incompare}
\begin{proof}
	We first show that there exists an instance for which  \eqref{scut1} dominates \eqref{scut2}. Consider the graph, $G_1$, shown in Figure \ref{fig:S1} and let $K= 3$. Let $\set_1 = \{v_1, v_2,v_3,v_4, v_5\}$. Then $G_1[\set_1] = G_1$ does not have a DMDGP order due to Infeasibility Check \ref{inf:mindeg}, e.g., since $v_3$ cannot be in the initial clique and cannot have $K$ contiguous predecessors. 
	
	\begin{figure}[h]
		\begin{subfigure}{0.6\textwidth}
			\centering
			\begin{tikzpicture}[vertex/.style={circle,fill=white!80,draw,inner sep=0pt, minimum size=16pt, scale = 1},
			line width=1.2pt, scale = 1]
			\node[vertex] (a) at (0,2) {$v_1$};
			\node[vertex]  (b) at (2,2) {$v_2$};
			\node[vertex]  (c) at (4,2)  {$v_3$};
			\node[vertex]  (d) at (4,0) {$v_4$};
			\node[vertex]  (e) at (0,0) {$v_5$};
			%		\node[vertex]  (f) at (-2,0) {$v_6$};
			\draw[-] (a) -- (b)-- (c) (b)-- (d) --(a) -- (e) --(d);
			\end{tikzpicture}
			\caption{$G_1$} 
			\label{fig:S1}
		\end{subfigure}
		~	\begin{subfigure}{0.4\textwidth}
			\centering
			\begin{tikzpicture}[vertex/.style={circle,fill=white!80,draw,inner sep=0pt, minimum size=16pt, scale = 1},
			line width=1.2pt, scale = 1]
			\node[vertex] (a) at (0,2) {$u_1$};
			\node[vertex]  (b) at (2,2) {$u_2$};
			\node[vertex]  (c) at (4,2)  {$u_3$};
			\node[vertex]  (d) at (0,4) {$u_4$};
			\node[vertex]  (e) at (2,4) {$u_5$};
			\draw (a) edge[-] (b) (b) edge[-] (c)
			(b) edge[-] (d) (d) edge[-] (a) (e) edge[-] (b);
			\end{tikzpicture}
			\caption{$G_2$} 
			\label{fig:S2}
		\end{subfigure}
		\caption{Graphs used to compare valid inequalities \eqref{scut1} and \eqref{scut2}.}
		\label{fig:Incomparables}
	\end{figure}
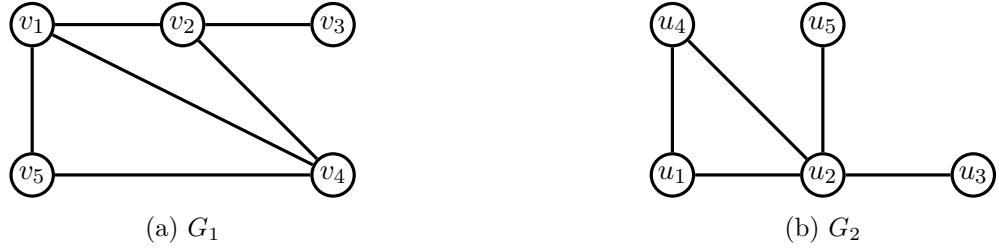

	To find the inequality \eqref{scut1}, we calculate $\max_{v \in \set_1} \deltamiss_{\set_1}(v) = \deltamiss_{\set_1}(v_3) = 3$ and find:
	\bsubeq
	\begin{alignat}{2}
	r_{max} - r_{min}&  \geq \max \{ |\set_1|,  \deltamiss_{\set_1} + \K\} \\
	r_{max} - r_{min} & \geq \max \{5,  3 + 3\} \\
	r_{max} - r_{min} & \geq 6
	\end{alignat}
	\esubeq
	The instance only has $5$ vertices in total, so this single inequality is enough to prove the instance does not have a DMDGP order.
	The largest stable set in the graph in Figure \ref{fig:S1} has cardinality two; using any of these stable sets as $\stableset_1$ makes the inequality \eqref{scut2}
	\bsubeq
	\begin{alignat}{2}
	r_{max} - r_{min}&  \geq (|\stableset_1|-1)(\K+1) \\
	r_{max} - r_{min} & \geq (2-1)(3+1) \\
	r_{max} - r_{min} & \geq 4
	\end{alignat}
	\esubeq
	which is not sufficient to prove the infeasibility of the instance. Thus \eqref{scut1} dominates \eqref{scut2} in this instance.\\
	We now show that there exists an instance for which  \eqref{scut2} dominates \eqref{scut1}. Consider the graph, $G_2$, shown in Figure \ref{fig:S2} and let $K= 3$. 
	A maximum stable set in $G_2$ is $\stableset_2 =\{u_1, u_3, u_5\}$, making the inequality \eqref{scut2}:
	\bsubeq
	\begin{alignat}{2}
	r_{max} - r_{min}&  \geq (|\stableset_2|-1)(\K+1) \\
	r_{max} - r_{min} & \geq (3-1)(3+1) \\
	r_{max} - r_{min} & \geq 8
	\end{alignat}
	\esubeq
	Setting $\set_2 = \stableset_2$, then $G[\set_2]$ does not have a DMDGP order due to Infeasibility Check \ref{inf:mindeg}.%, e.g., since $u_3$ cannot be in the initial clique and cannot have $K$ adjacent immediate predecessors. 
	We calculate $\max_{u \in \set_2} \deltamiss_{\set_2}(u) = 2$ and find
	\bsubeq
	\begin{alignat}{2}
	r_{max} - r_{min}&  \geq \max \{ |\set_2|,  \deltamiss_{\set_2} + \K\} \\
	r_{max} - r_{min} & \geq \max \{ 3,  2+ 3\} \\
	r_{max} - r_{min} & \geq 5
	\end{alignat}
	\esubeq
	Thus \eqref{scut2} dominates \eqref{scut1}, and we may conclude that \eqref{scut1} and \eqref{scut2} are incomparable.

\end{proof}

\subsection{Proof of Proposition \ref{wheel}} 
\label{proof_wheel}

\begin{proof} Let the vertices in the wheel be indexed as in Figure \ref{fig:W7}, i.e., the centre vertex has index 0, while the ones in the peripheral cycle are indexed from $[1,n]$ counter clock-wise starting at an arbitrary one.
	\BI
	\I \textit{Case 1:} When $n$ is odd, a maximum stable set in $W_n$ is $\{ 1, 3, \hdots, n-2 \} $ with size $\frac{n-1}{2}$. For the right-hand side of the inequality in Infeasibility Check \ref{inf:maxSS} to hold, we need 
	\[\frac{n-1}{2} \geq\frac{n}{\K+1} + 1 \]
	\[\frac{n}{2} - \frac{n}{\K+1} \geq \frac{1}{2}\]
	\[n \left(\frac{1}{2} - \frac{1}{\K+1}\right) \geq \frac{1}{2}\]
	\[n  \geq  \frac{\K+1}{\K-1}. \]
	\I \textit{Case 2:} Similarly, when $n$ is even, a maximum stable set in $W_n$ is $\{ 1, 3, \hdots, n-3 \} $ with size $\frac{n}{2}$. So we need
	\[\frac{n}{2} \geq\frac{n}{\K+1} + 1 \]
	\[n \left(\frac{1}{2} - \frac{1}{\K+1}\right) \geq 1\]
	\[n  \geq 2 \frac{\K+1}{\K-1}. \]
	\EI
	Thus the inequalities hold.
\end{proof}

\section{Example}
\label{Exapp}

\begin{ex}
	We will demonstrate the strength of symmetry breaking on the graph in Figure \ref{fig:DMDGPsymm} with $\K =2$. 
	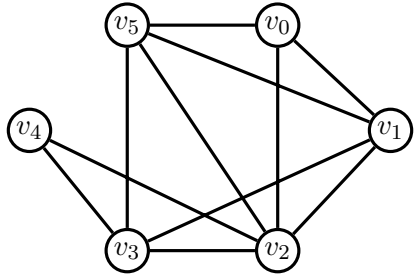
\begin{figure}[h]
		\centering
		\begin{tikzpicture}[main_node/.style={circle,fill=white!80,draw,inner sep=0pt, minimum size=16pt, scale=1},
		line width=1.2pt, scale=1]
		\node[main_node] (v0) at (2,3) {$v_0$};
		\node[main_node] (v1) at (3.5,1.6) {$v_1$};
		\node[main_node] (v2) at (2,0) {$v_2$};
		\node[main_node] (v3) at (0,0) {$v_3$};
		\node[main_node] (v4) at (-1.3,1.6) {$v_4$};
		\node[main_node] (v5) at (0, 3) {$v_5$};
		\draw[-] (v0) -- (v1) -- (v2) -- (v3) -- (v4) -- (v2) -- (v0) -- (v5) -- (v3) -- (v1) -- (v5) -- (v2);
		\end{tikzpicture}
		\caption{ A graph instance which is feasible for CTOP with $\K =2$.} \label{fig:DMDGPsymm}
	\end{figure}
	
	This instance has $12$ feasible DMDGP orders:
	\begin{alignat*}{3}
	\Lorder v_4, v_3, v_2, v_1, v_0, v_5 \Rorder &\quad& \Lorder v_5, v_0, v_1, v_2, v_3, v_4 \Rorder &\quad& \Lorder v_4, v_3, v_2, v_5, v_0, v_1 \Rorder \\
	\Lorder v_1, v_0, v_5, v_2, v_3, v_4 \Rorder &\quad& \Lorder v_4, v_2, v_3, v_1, v_5, v_0 \Rorder &\quad& \Lorder v_0, v_5, v_1, v_3, v_2, v_4 \Rorder\\
	\Lorder v_4, v_3, v_2, v_1, v_5, v_0 \Rorder &\quad& \Lorder v_0, v_5, v_1, v_2, v_3, v_4 \Rorder &\quad& \Lorder v_4, v_2, v_3, v_5, v_1, v_0 \Rorder\\
	\Lorder v_0, v_1, v_5, v_3, v_2, v_4 \Rorder &\quad& \Lorder v_4, v_3, v_2, v_5, v_1, v_0  \Rorder &\quad& \Lorder v_0, v_1, v_5, v_2, v_3, v_4 \Rorder
	\end{alignat*}
	We begin by noticing that $\vertexset^{\degr[2,2]} =  \{v_4\}$	so we fix $r_{v_4} = 0$ and eliminate half the orders, leaving the orders:
	\begin{alignat*}{3}
	\Lorder v_4, v_3, v_2, v_1, v_0, v_5 \Rorder &\quad& \Lorder v_4, v_2, v_3, v_1, v_5, v_0 \Rorder  &\quad& \Lorder v_4, v_3, v_2, v_5, v_0, v_1 \Rorder \\
	\Lorder v_4, v_3, v_2, v_1, v_5, v_0 \Rorder &\quad& \Lorder v_4, v_3, v_2, v_5, v_1, v_0  \Rorder  &\quad& \Lorder v_4, v_2, v_3, v_5, v_1, v_0 \Rorder
	\end{alignat*}
	There are no stable sets meeting Symmetry Breaking Condition \ref{sym:ss}, and the only clique meeting Symmetry Breaking Condition \ref{sym:clique} is $\{v_1,v_5\}$. Thus we enforce $r_{v_1} < r_{v_5}$ and eliminate another three orders, leaving three remaining orders:
	\begin{alignat*}{3}
	\Lorder v_4, v_3, v_2, v_1, v_0, v_5 \Rorder &\quad&  \Lorder v_4, v_3, v_2, v_1, v_5, v_0 \Rorder  &\quad& \Lorder v_4, v_3, v_2, v_5, v_0, v_1 \Rorder.
	\end{alignat*}
	Since we have found a clique in Symmetry Breaking Condition \ref{sym:clique}, we first examine Symmetry Breaking Condition \ref{sym:ex_clique}. Beginning with $\clique = \{v_1,v_5\}$ and $v = v_2$, we have 
	\[
	(\neighb(v_2) \cup \{v_2\} )\setminus (\neighb(\{v_1,v_5\}) \cup \{v_1,v_5\})  = \{v_0,v_1,v_2,v_3,v_4,v_5\} \setminus \{v_0,v_1,v_2,v_3,v_5\} =\{v_4\}
	\]
	so $w= v_4$ and we can add the following logical constraints:
	\[ |\rankvar_{v_2} - \rankvar_{v_4}| \geq 3 \implies \rankvar_{v_2} < \rankvar_{v_5}\]
	\[ |\rankvar_{v_2} - \rankvar_{v_4}| \geq 3 \implies \rankvar_{v_2} < \rankvar_{v_1}\]
	In fact, the latter suffices since we have already added $ \rankvar_{v_1} < \rankvar_{v_5}$. Unfortunately this does not remove any solutions from the pool. We now try $\clique=\{v_0\}$ and $v=v_5$. In this case we have $w = v_3$ and add
	\[ |\rankvar_{v_5} - \rankvar_{v_3}| \geq 3 \implies \rankvar_{v_5} < \rankvar_{v_0}\]
	which removes a further order, yielding the remaining orders:
	\begin{alignat*}{3}
	\Lorder v_4, v_3, v_2, v_1, v_5, v_0 \Rorder  &\quad& \Lorder v_4, v_3, v_2, v_5, v_0, v_1 \Rorder .
	\end{alignat*}
	Finally, we extend $\clique=\{v_3\}$ using $v=2$, giving $w=v_0$ and the logical constraint:
	\[ |\rankvar_{v_2} - \rankvar_{v_0}| \geq 3 \implies \rankvar_{v_2} < \rankvar_{v_3}\]
	Thus symmetry breaking has reduced the solution space to a single DMDGP order:
	\begin{alignat*}{3}
	\Lorder v_4, v_3, v_2, v_5, v_0, v_1 \Rorder 
	\end{alignat*}
\end{ex}

\section{Summary of the paper} 
\label{app:summary}
Table \ref{table:DMDGPsummary} provides a summary of the models from the literature as well as all of our proposed formulations and structural insights.

\afterpage{
\begin{landscape}
	\centering
	\renewcommand\arraystretch{0.65}
	\small
\setlength{\LTleft}{-16pt}	
\begin{longtable}{lllllll}
	\caption{Summary of CTOP formulations and structural findings.}
	\label{table:DMDGPsummary}\\
\toprule
\multicolumn{7}{l}{\textbf{DMDGP:} Given a graph $\graph = (\vertexset, \edgeset)$ and integer $\K> 0$, minimally, a DMDGP order is a series of $(\K + 1)$-cliques which overlap by at least $\K$ vertices.} \\
\midrule
\cmidrule(l{0em}r{1em}){1-7}
{\cellcolor{gray!50!white}\textbf{FORMULATIONS}}  & Variables & Domain & Number & Constraints & Number & Comments \\
{\cellcolor{gray!20!white}\textsc{Literature}}   \\
Vertex-rank $\IPVR$ & $\vrvar_{\vertind \rankind}$ & $\{0,1\}$ & $n^2$ & 1-1 assignment & $2n$ \\
 &  &  &  & clique  & $n^2$ \\
Clique Digraph $\IPCD$ & $\arcvar_{\orderclique_i \orderclique_j}$ & $\{0,1\}$ & $|\orderedcliques|^2$ & select initial and last clique & 2 & Enumerate ordered  \\
 & $\initcliquevar_{\orderclique_j}$ & $\{0,1\}$ & $|\orderedcliques|$ & flow balance & $|\orderedcliques|$&$(\K+1)$-cliques \\
 & $\lastcliquevar_{\orderclique_j}$ & $\{0,1\}$ & $|\orderedcliques|$ & successor & $|\orderedcliques|$ \\
 & $\predvar_{uv}$ & $\{0,1\}$ & $|\vertexset|^2 $& vertex covering & $|\vertexset|$ \\
 &  &  &  & precedence & $|\vertexset|^2+ |\vertexset|^3$ \\
 &  &  &  & clique ordering & $(2\K+1)\times|\orderedcliques|^2$ \\
Clique Digraph Relax.\ $\Unorderedrelax$ & $\arcvar_{\orderclique_i \orderclique_j}$ & $\{0,1\}$ & $|\orderedcliques|^2$ & select initial and last clique & 2 & Enumerate $(\K+1)$-cliques \\
 & $\initcliquevar_{\orderclique_j}$ & $\{0,1\}$ & $|\orderedcliques|$ & flow balance & $|\orderedcliques|$ &If infeasible, \\
 & $\lastcliquevar_{\orderclique_j}$ & $\{0,1\}$ & $|\orderedcliques|$ & successor & $|\orderedcliques|$ &no DMDGP order \\
 & $\predvar_{uv}$ & $\{0,1\}$ & $n^2$ & vertex covering & $|\vertexset|$ &o.w., may or may not have  \\
 &  &  &  & precedence & $|\vertexset|^2+ |\vertexset|^3$&DMDGP order \\
 &  &  &  & relaxed clique ordering & $2|\orderedcliques|^2 +|\orderedcliques|$ \\
 &  &  &  & no. times vertex in clique & $|\vertexset|$ \\
{\cellcolor{gray!20!white}\textsc{New}} \\
CP Rank $\rankDMDGP$ & $\rankvar_\vertind$ & $[n-1]$ & $n$ & AllDifferent & 1 \\
 &  &  &  & clique & $|\vertexset|^2$ \\
CP Vertex $\vertexDMDGP$ & $\vertvar_\rankind$ & $[|\vertexset|-1]$ & $n$ & AllDifferent & 1 \\
 &  &  &  & clique & $|\vertexset|^2$ \\
CP Combined $\combDMDGP$ & $\rankvar_\vertind$ & $[n-1]$ & $n$ & AllDifferent and inverse& 3 & Combines CP Rank and  \\
 & $\vertvar_\rankind$ & $[|\vertexset|-1]$ & $n$ & clique & $2|\vertexset|^2$& CP Vertex\\
 \midrule
{\cellcolor{gray!50!white}\textbf{STRUCTURAL FINDINGS}} &&&&  {\cellcolor{gray!20!white}\textsc{Symmetry Breaking}}  \\
{\cellcolor{gray!20!white} \textsc{Infeasibility Checks}} &&&& Arbitrary &  \multicolumn{2}{c}{ $r_{v_1} < r_{v_2}$ }\\
Minimum Degree & \multicolumn{3}{c}{$\degr(\vertind) < \K$} & Degree $\K$ & \multicolumn{2}{c}{ $r_{v_i} =0$ and $r_{v_j} =n-1$} \\
Minimum Edges & \multicolumn{3}{c}{ $|\edgeset| < \left (|\vertexset| - \frac{1}{2} \right )\K - \frac{1}{2} \K^2$ } & 	Stable Set Same Neighbours &\multicolumn{2}{c}{ $\rankvar_{\vertind_1} < \rankvar_{\vertind_2} < \cdots < \rankvar_{\vertind_k}$} \\
UB on Small Deg.\ Vertices &\multicolumn{3}{c}{ $\left |\vertexset^{\degr[\K, \K+\delta]} \right | > 2(\delta +1) +1$ } & Clique Same Neighbours &\multicolumn{2}{c}{ $\rankvar_{\vertind_1} < \rankvar_{\vertind_2} < \cdots < \rankvar_{\vertind_k}$ }\\
LB on Large Deg.\ Vertices &\multicolumn{3}{c}{ $\left |\vertexset^{\degr[2\K, n-1]} \right | \leq n-(2\K+1)$} & Extended Stable Set &\multicolumn{2}{c}{ $|\rankvar_\vertind - \rankvar_w| \geq K+1 \implies \rankvar_\vertind < \rankvar_u $} \\
Max Stable Set &\multicolumn{3}{c}{ $\stableset > \frac{n}{\K+1} + 1 $} &Extended Clique & \multicolumn{2}{c}{ $|\rankvar_\vertind - \rankvar_w| \geq K+1 \implies \rankvar_\vertind < \rankvar_u $}\\
{\cellcolor{gray!20!white} \textsc{Domain Reduction} }&&&& {\cellcolor{gray!20!white} \textsc{Valid Inequalities}} \\
Small Deg.\ Vertices & \multicolumn{3}{c}{$[\degr(\vertind)-\K] \cup [n-1-(\degr(\vertind)-\K), n-1]$} & Vertex Subset &\multicolumn{2}{c}{ $r_{max} - r_{min} \geq |\set|$ }  \\
Neighb.\ Small Deg.\ Vertices & \multicolumn{3}{c}{ $[ \degr(\vertind^*)] \cup [n-1-\degr(\vertind^*), n-1] $ } & Improved Vertex Subset & \multicolumn{2}{c}{$r_{max} - r_{min}  \geq \max \{ |\set|, \deltamiss_\set+K\}$ }\\
&&&& Stable Set & \multicolumn{2}{c}{$r_{max} - r_{min}   \geq (|\stableset| -1)(\K+1) $}\\
\bottomrule
\end{longtable}
\end{landscape}
\clearpage % flush the float
\restoregeometry} 
\newpage
\section{DMDGP Results} 
\label{app:DMDGPApp}
 
 We provide the following tables:
 \BI
 \I Table \ref{table:smallExact} compares the integer programming formulations from the literature \cite{cassioli2015}, namely the vertex-rank formulation $\IPVR$ and the clique digraph formulation $\IPCD$, with the newly proposed constraint programming formulations, namely the rank-based primal formulation $\rankDMDGP$, the vertex-based dual formulation $\vertexDMDGP$ and the combined formulation $\combDMDGP$ on the small instances for a variety of densities {\crev with $\K =3$}.
 \I Table \ref{table:largeCP} compares the CP formulations {\crev $\rankDMDGP$,} $\vertexDMDGP$, and $\combDMDGP$ on the medium-sized instances for low to medium densities {\crev with $\K =3$}.
 \I Table \ref{table:CDrelax} compares the IP Clique Digraph formulations $\IPCD$ and $\Unorderedrelax$ on small instances {\crev with $\K =3$}.
 \I {\crev Tables \ref{table:prot30},  \ref{table:prot40},  \ref{table:prot50}, and  \ref{table:prot60}  compare the $\IPVR$, $\vertexDMDGP$, and $\combDMDGP$ formulations with $\K=3$ on pseudo-protein instances with $30,40,50,$ and $60$  vertices respectively.}
 
 \I Table \ref{table:enhance} compares the $\combDMDGP$ formulation {\crev with $\K =3$}, with and without the additions from the structural analysis, on small instances. It also gives the rule that was applied and to how many vertices or sets using the following conventions:
 	\BI
 	\I $[\text{Inf} r]$
refers to Infeasibility Check $r$, either 1 or 3.
 	\I $[\text{DR} r]$ refers to Domain Reduction Rule 1 which has been applied to $r$ vertices.
 	\I $[\text{E}\stableset s]$ refers to Symmetry Breaking Condition 4 which has been applied to $s$ stable sets.
 	 	\I $[\text{E}\clique c]$ refers to Symmetry Breaking Condition 5 which has been applied to $c$ cliques
 	\I $[\text{Arb}]$ refers to an arbitrary ordering on two vertices, as in Symmetry Breaking Condition 6.
 \EI 
 We note that Infeasibility Checks 2 and 4, Domain Reduction Rule 2, and Symmetry Breaking Conditions 1, 2, 3 are excluded from the table since they were never applied.	
 \I{ \crev Table \ref{table:VI} compares $\rankDMDGP$ and $\combDMDGP$ with valid inequalities in the form (7) and (8) and without valid inequalities. Valid inequalities are generated for every stable set in each instance and added before solving.}
 \I{ \crev Table \ref{table:K4} compares the vertex-rank formulation $\IPVR$ with CP formulations $\rankDMDGP$,  $\vertexDMDGP$, and $\combDMDGP$ on the small instance with $\K=4$.} 
 \I{ \crev Table \ref{table:K5} compares the vertex-rank formulation $\IPVR$ with CP formulations $\rankDMDGP$,  $\vertexDMDGP$, and $\combDMDGP$ on the small instance with $\K=5$.} 
 \EI
For the instances, we have
\BI
\I ``Status": Feasibility status of the instance; ``Feas." and ``Infeas." if it is proven to be feasible and infeasible by any of the methods, respectively, ``Unsol." otherwise as it is not solved by any method.
\EI 
{\crev In particular, for random instances we have
\BI
\I ``$n$": The number of vertices in the input graph.
\BI
\I $n \in \{ 20, 25, \hdots, 60\}$ for the small instances.
\I $n \in \{ 65, 70, \hdots, 100\}$ for the medium-sized instances.
\EI 
\I ``$\density$": The edge density of the input graph.
\BI
\I $\density \in \{ 0.3,0.5,0.7\}$ for the small instances.
\I $\density \in \{ 0.2,0.3,0.4,0.5\}$ for the medium-sized instances.
\EI
\I ``Inst.": The assigned instance number from $\{1,2,3\}$ for each $(n,\density)$ combination.
\EI
For pseudo-protein instances we have
\BI
\I ``Protein Inst.": the Protein Data Bank protein which was modified to create the instance.
\I ``Ratio": the approximate average edge to vertex ratio of the instance.
\EI

}
We also have
\BI
\I ``Time": Solution time in seconds if the instance is solved in the given time limit, ``TL" if the instance hit the time limit, ``MEM" if the instance hit the memory limit.
\I ``BB Nodes": The number of branch-and-bound nodes explored (for the IP formulations); exact if it is less than one thousand, lower bound rounded to the closest million otherwise where a single decimal point is used up to between one million for a better accuracy.
\I ``Ch.Pts.": The number of choice points (for the CP formulations); the number convention is the same as the ``BB Nodes".
\I ''Solved": indicates if a feasible solution was found.
\I ``DMDGP found": ``Yes" if the feasible solution is a DMDGP order, ``No" otherwise.
\I ``\# cliques": the number of ordered and unordered cliques for $\IPCD$ and $\Unorderedrelax$ respectively.
\EI

\begin{landscape}
\renewcommand\arraystretch{0.8}
\begin{center}
\small
\begin{longtable}{cccl rr rr rr rr rr}
\caption{Integer programming and constraint programming results for small instances}
\label{table:smallExact} \\
\toprule
 &  &  & &   \multicolumn{2}{c}{$\IPVR$} & \multicolumn{2}{c}{$\IPCD$} & \multicolumn{2}{c}{$\rankDMDGP$} & \multicolumn{2}{c}{$\vertexDMDGP$}  & \multicolumn{2}{c}{$\combDMDGP$} \\
\cmidrule(lr){5-6}
\cmidrule(lr){7-8}
\cmidrule(lr){9-10}
\cmidrule(lr){11-12}
\cmidrule(lr){13-14}
$n$ & $\density$ & Inst. & Status &  Time &  BB Nodes &  Time &  BB Nodes & Time & Ch.Pts. & Time & Ch.Pts. & Time & Ch.Pts. \\
\midrule
20 & 0.3 & 1 & Infeas. & 41.43 & 5981 & \textBF{0.00} & 0 & 0.16 & 8937 & 0.10 & 1987 & 0.01 & 72 \\ 
&  & 2 & Infeas. & 0.01 & 0 & \textBF{0.00} & 0 & 0.13 & 7286 & 0.01 & 130 & 0.00 & 17 \\ 
&  & 3 & Infeas. & 29.08 & 5188 & \textBF{0.00} & 0 & 0.15 & 8255 & 0.01 & 122 & 0.00 & 20 \\*[0.2cm]
& 0.5 & 1 & Infeas. & 132.53 & 52865 & 4.13 & 0 & \textBF{0.21} & 10850 & 0.80 & 9705 & 0.91 & 9423 \\ 
&  & 2 & Infeas. & 156.95 & 79733 & 5.66 & 0 & \textBF{0.13} & 7408 & 0.65 & 8402 & 0.29 & 3909 \\
\bottomrule
\end{longtable}
\end{center}
\end{landscape}
%\newpage
%\begin{landscape}
\renewcommand\arraystretch{0.8}
\begin{center}
\small
\begin{longtable}{cccl rr rr rr}
\caption{Constraint programming results for medium-sized instances}
\label{table:largeCP} \\
\toprule 
 &  &  & & \multicolumn{2}{c}{$\rankDMDGP$} &  \multicolumn{2}{c}{$\vertexDMDGP$} & \multicolumn{2}{c}{$\combDMDGP$}  \\
\cmidrule(lr){5-6}
\cmidrule(lr){7-8}
\cmidrule(lr){9-10}
$n$ & $\density$ & Inst. & Status & Time & Ch.Pts. & Time & Ch.Pts. & Time & Ch.Pts. \\
\midrule
65 & 0.2 & 1 & Infeas. & TL & 265M & 4.60 & 12430 & \textBF{2.46} & 3781 \\
&  & 2 & Infeas. & TL & 266M & 5.01 & 12963 & \textBF{3.50} & 4263 \\
&  & 3 & Infeas. & TL & 267M & 8.09 & 16647 & \textBF{2.49} & 3574 \\*[0.18cm]
& 0.3 & 1 & Infeas. & TL & 297M & 352.77 & 428908 & \textBF{86.21} & 79144 \\
&  & 2 & Infeas. & TL & 294M & 2659.91 & 3.8M & \textBF{22.35} & 23589 \\
\bottomrule
\end{longtable}
\end{center}
%\end{landscape}

\newpage
\begin{landscape}
\renewcommand\arraystretch{0.8}
\begin{center}
	\small
	\begin{longtable}{cccl rr rr rr rr}
					\caption{Results for $\IPCD$ and $\Unorderedrelax$ on small instances.}
		\label{table:CDrelax} \\
		\toprule
 &  &  &  &\multicolumn{2}{c}{ $\Unorderedrelax$} & \multicolumn{2}{c}{\# Cliques} &  \multicolumn{2}{c}{Time} &  \multicolumn{2}{c} {BB Nodes} \\
	 \cmidrule(lr){5-6}
 \cmidrule(lr){7-8}
 \cmidrule(lr){9-10}
  \cmidrule(lr){11-12}
% \cmidrule(lr){12-14}
% \cmidrule(lr){15-17}
$n$ & $\density$ & Inst. & Status & Solved & DMDGP found & $\IPCD$ & $\Unorderedrelax $ &  $\IPCD$ & $\Unorderedrelax$ &  $\IPCD$ & $\Unorderedrelax$ \\
\midrule
20 & 0.3 & 1 & Infeas. & - & - & 168 & 7 & \textBF{0.00} & 0.00 & 0 & 0 \\
 &  & 2 & Infeas. & - & - & 48 & 2 & \textBF{0.00} & 0.00 & 0 & 0 \\
 &  & 3 & Infeas. & - & - & 96 & 4 & \textBF{0.00} & 0.00 & 0 & 0 \\*[0.18cm]
 & 0.5 & 1 & Infeas. & - & - & 1656 & 69 & 4.13 & \textBF{2.08} & 0 & 0 \\
 &  & 2 & Infeas. & - & - & 1920 & 80 & 5.66 & \textBF{1.83} & 0 & 0 \\
\bottomrule
\end{longtable}
\end{center}
\end{landscape}

	\renewcommand\arraystretch{0.8}
	\vspace*{-1.28cm}
	\centering
	\begin{longtable}{crl rr rr rr}
		\caption{Integer programming and constraint programming results for pseudo-protein instances with $30$ vertices.}
		\label{table:prot30} \\
		\toprule
		&  &  &   \multicolumn{2}{c}{$\IPVR$} & \multicolumn{2}{c}{$\vertexDMDGP$}  & \multicolumn{2}{c}{$\combDMDGP$} \\
		\cmidrule(lr){4-5}
		\cmidrule(lr){6-7}
		\cmidrule(lr){8-9}
	Protein Inst. & Ratio & Status &  Time &  BB Nodes & Time & Ch.Pts. & Time & Ch.Pts. \\
	\midrule
	1a70 & 3.2 & Infeas. & 4604.21 & 52219 & 1.08 & 6709 & \textBF{0.41} & 2859 \\
	& 3.6 & Infeas. & 2225.23 & 39693 & 4.18 & 23385 & \textBF{0.91} & 4328 \\
	& 4 & Infeas. & TL & 56015 & 4.32 & 38727 & \textBF{0.80} & 3867 \\
	\midrule
	1bpm & 3.2 & Infeas. & 322.34 & 8808 & 1.65 & 8308 & \textBF{0.35} & 2238 \\
	& 3.6 & Infeas. & TL & 56680 & 3.41 & 15735 & \textBF{0.69} & 3013 \\
\bottomrule
	\end{longtable}

\newpage
	\renewcommand\arraystretch{0.8}
	\vspace*{-1.28cm}
	\centering
	\begin{longtable}{crl rr rr rr rr}
		\caption{Integer programming and constraint programming results for pseudo-protein instances with $40$ vertices.}
		\label{table:prot40} \\
		\toprule
&  &  &   \multicolumn{2}{c}{$\IPVR$} & \multicolumn{2}{c}{$\vertexDMDGP$}  & \multicolumn{2}{c}{$\combDMDGP$} \\
\cmidrule(lr){4-5}
\cmidrule(lr){6-7}
\cmidrule(lr){8-9}
Protein Inst. & Ratio & Status &  Time &  BB Nodes & Time & Ch.Pts. & Time & Ch.Pts. \\
		\midrule
		1a70 & 3.2 & Infeas. & TL & 10204 & 1.67 & 8308 & \textBF{0.61} & 2950 \\
		& 3.6 & Infeas. & TL & 25380 & 5.75 & 30493 & \textBF{0.83} & 3628 \\
		\midrule
		1bpm & 3.2 & Infeas. & TL & 16123 & 2.50 & 8143 & \textBF{0.79} & 4058 \\
		& 3.6 & Infeas. & TL & 19203 & 5.88 & 19069 & \textBF{0.72} & 3902 \\
		\midrule
		1crn & 3.2 & Infeas. & 4224.48 & 10435 & 1.79 & 8465 & \textBF{0.94} & 4454 \\
\bottomrule
	\end{longtable}

\newpage
	\renewcommand\arraystretch{0.8}
	\vspace*{-1.28cm}
	\centering
	\begin{longtable}{crl rr rr rr rr}
		\caption{Integer programming and constraint programming results for pseudo-protein instances with $50$ vertices.}
		\label{table:prot50} \\
		\toprule
&  &  &   \multicolumn{2}{c}{$\IPVR$} & \multicolumn{2}{c}{$\vertexDMDGP$}  & \multicolumn{2}{c}{$\combDMDGP$} \\
\cmidrule(lr){4-5}
\cmidrule(lr){6-7}
\cmidrule(lr){8-9}
Protein Inst. & Ratio & Status &  Time &  BB Nodes & Time & Ch.Pts. & Time & Ch.Pts. \\
\midrule
1a70 & 3.2 & Infeas. & TL & 7146 & \textBF{1.28} & 6806 & 2.03 & 4484 \\
& 3.6 & Infeas. & TL & 15102 & 9.97 & 24153 & \textBF{2.03} & 3884 \\
& 4 & Infeas. & TL & 10292 & \textBF{7.11} & 24801 & 15.99 & 20542 \\
\midrule
1bpm & 3.2 & Infeas. & TL & 7896 & 3.78 & 7000 & \textBF{2.62} & 6094 \\
& 3.6 & Infeas. & TL & 9105 & 7.76 & 13444 & \textBF{3.22} & 5509 \\
\bottomrule
\end{longtable}

\newpage
	\renewcommand\arraystretch{0.8}
	\vspace*{-1.28cm}
	\centering
	\begin{longtable}{ccl rr rr rr rr}
		\caption{Integer programming and constraint programming results for pseudo-protein instances with $60$ vertices.}
\label{table:prot60} \\
		\toprule
&  &  &   \multicolumn{2}{c}{$\IPVR$} & \multicolumn{2}{c}{$\vertexDMDGP$}  & \multicolumn{2}{c}{$\combDMDGP$} \\
\cmidrule(lr){4-5}
\cmidrule(lr){6-7}
\cmidrule(lr){8-9}
Protein Inst. & Ratio & Status &  Time &  BB Nodes & Time & Ch.Pts. & Time & Ch.Pts. \\
\midrule
1a70 & 3.2 & Infeas. & TL & 2162 & \textBF{1.86} & 5387 & 4.60 & 7243 \\
& 3.6 & Infeas. & TL & 1569 & 14.32 & 24469 & \textBF{3.62} & 5593 \\
& 4 & Infeas. & TL & 4112 & 63.92 & 208576 & \textBF{4.71} & 5972 \\
\midrule
1bpm & 3.2 & Infeas. & TL & 2470 & 12.52 & 17783 & \textBF{4.81} & 8603 \\
& 3.6 & Infeas. & TL & 4552 & 16.45 & 22124 & \textBF{4.53} & 7544 \\
\bottomrule
\end{longtable}

\begin{landscape}
\renewcommand\arraystretch{0.8}
\begin{center}
	\small
\begin{longtable}{cccl  rr rrl}
	\caption{Results for structural findings with $\combDMDGP$ on small instances. In the structural findings rule column we use the convention (all)/(inf.\ checks \& dom.\ red.)/(inf. checks \& sym.\ break.). We have $[\text{Inf} ]$ for Infeasibility Check 1 or 3, $[\text{DR}]$ for Domain Reduction Rule 1, $[\text{E}\stableset]$ and $[\text{E}\clique]$, and $[\text{Arb}]$ for  Symmetry Breaking Condition 4, 5, and 6 respectively.}
	\label{table:enhance} \\
	\toprule 
 &  &  &  &  \multicolumn{2}{c}{$\combDMDGP$} &  \multicolumn{3}{c}{Structural Findings} \\
 \cmidrule(lr){5-6}
 \cmidrule(lr){7-9}
$n$ & $\density$ & Inst. & Status & Time & Ch.Pts. & Time & Ch.Pts. & Rule \\
\midrule
20 & 0.3 & 1 & Infeas. & 0.01 & 72 & \textBF{0.00} & 0 & [Inf 3] \\
&  & 2 & Infeas. & 0.00 & 17 & \textBF{0.00} & 0 & [Inf 1] \\
&  & 3 & Infeas. & 0.00 & 20 & \textBF{0.00} & 0 & [Inf 3] \\*[0.18cm]
& 0.5 & 1 & Infeas. & 0.91 & 9423 & 0.76/1.12/\textBF{0.57} & 6691/9423/6691 & [E$\stableset$ 8], [E$\clique$ 9] \\
&  & 2 & Infeas. & 0.29 & 3909 & \textBF{0.22}/0.27/0.39 & 2282/2857/5041 & [DR1], [E$\stableset$7] , [E$\clique$10] \\
\bottomrule
\end{longtable}
\end{center}
\end{landscape}
\begin{landscape}
\renewcommand\arraystretch{0.8}
\begin{center}
\small
\begin{longtable}{cccl rr rr rr rr rr rr}
\caption{$\rankDMDGP$ and $\combDMDGP$ results with and without valid inequalities on small instances.}
\label{table:VI} \\
\toprule
 &  &  & &   \multicolumn{2}{c}{$\rankDMDGP$} & \multicolumn{2}{c}{$\rankDMDGP$} & \multicolumn{2}{c}{$\rankDMDGP$} & \multicolumn{2}{c}{$\combDMDGP$}  & \multicolumn{2}{c}{$\combDMDGP$} & \multicolumn{2}{c}{$\combDMDGP$}\\
  &  &  & &   \multicolumn{2}{c}{No Valid Inequ. } &\multicolumn{2}{c}{Valid Inequ. (7)}  & \multicolumn{2}{c}{Valid Inequ. (8)}  & \multicolumn{2}{c}{No Valid Inequ. }  & \multicolumn{2}{c}{Valid Inequ. (7)} & \multicolumn{2}{c}{Valid Inequ. (8)} \\
\cmidrule(lr){5-6}
\cmidrule(lr){7-8}
\cmidrule(lr){9-10}
\cmidrule(lr){11-12}
\cmidrule(lr){13-14}
\cmidrule(lr){15-16}
$n$ & $\density$ & Inst. & Status & Time & Ch.Pts. & Time & Ch.Pts. & Time & Ch.Pts. & Time & Ch.Pts. & Time & Ch.Pts. & Time & Ch.Pts. \\
\midrule
20 & 0.3 & 1 & Infeas. & 0.16 & 8937 & \textBF{0.00} & 0 & 0.00 & 0 & 0.01 & 72 & 0.00 & 0 & 0.00 & 0 \\
&  & 2 & Infeas. & 0.13 & 7286 & \textBF{0.00} & 0 & 0.00 & 0 & 0.00 & 17 & 0.00 & 0 & 0.00 & 0 \\
&  & 3 & Infeas. & 0.15 & 8255 & 0.00 & 0 & \textBF{0.00} & 0 & 0.00 & 20 & 0.00 & 0 & 0.00 & 0 \\
& 0.5 & 1 & Infeas. & 0.21 & 10850 & 0.32 & 14216 & 0.30 & 1437 & 0.91 & 9423 & 0.30 & 2483 & \textBF{0.05} & 43 \\
\bottomrule
\end{longtable}
\end{center}
\end{landscape}
\begin{landscape}
\renewcommand\arraystretch{0.8}
\vspace*{-1.28cm}
\centering
	\begin{longtable}{cccl rr rr rr rr}
		\caption{Integer programming and constraint programming results for small instances with $\K=4$.}
		\label{table:K4} \\
	\toprule
	&  &  & &   \multicolumn{2}{c}{$\IPVR$} &  \multicolumn{2}{c}{$\rankDMDGP$} & \multicolumn{2}{c}{$\vertexDMDGP$}  & \multicolumn{2}{c}{$\combDMDGP$} \\
	\cmidrule(lr){5-6}
	\cmidrule(lr){7-8}
	\cmidrule(lr){9-10}
	\cmidrule(lr){11-12}
	$n$ & $\density$ & Inst. & Status &  Time &  BB Nodes &  Time & Ch.Pts. & Time & Ch.Pts. & Time & Ch.Pts. \\
	\midrule
20 & 0.3 & 1 & Infeas. & 0.02 & 0 & 0.05 & 4119 & 0.13 & 2148 & \textBF{0.01} & 41 \\
&  & 2 & Infeas. & 0.02 & 0 & 0.04 & 3743 & 0.01 & 130 & \textBF{0.00} & 15 \\
&  & 3 & Infeas. & 14.63 & 1766 & 0.05 & 4700 & 0.01 & 122 & \textBF{0.00} & 18 \\*[0.18cm]
& 0.5 & 1 & Infeas. & 151.22 & 35904 & \textBF{0.04} & 3907 & 0.98 & 10887 & 0.18 & 2769 \\
&  & 2 & Infeas. & 150.63 & 43948 & \textBF{0.06} & 4784 & 0.54 & 6671 & 0.15 & 2283 \\
\bottomrule
\end{longtable}
\end{landscape}

\begin{landscape}
	\renewcommand\arraystretch{0.8}
	\vspace*{-1.28cm}
	\centering
	\begin{longtable}{cccl rr rr rr rr}
		\caption{Integer programming and constraint programming results for small instances with $\K=5$.}
		\label{table:K5} \\
		\toprule
		&  &  & &   \multicolumn{2}{c}{$\IPVR$} &  \multicolumn{2}{c}{$\rankDMDGP$} & \multicolumn{2}{c}{$\vertexDMDGP$}  & \multicolumn{2}{c}{$\combDMDGP$} \\
		\cmidrule(lr){5-6}
		\cmidrule(lr){7-8}
		\cmidrule(lr){9-10}
		\cmidrule(lr){11-12}
		$n$ & $\density$ & Inst. & Status &  Time &  BB Nodes &  Time & Ch.Pts. & Time & Ch.Pts. & Time & Ch.Pts. \\
		\midrule
		20 & 0.3 & 1 & Infeas. & 0.02 & 0 & 0.04 & 2328 & 0.07 & 1259 & \textBF{0.01} & 33 \\
		&  & 2 & Infeas. & 0.02 & 0 & 0.03 & 1999 & 0.02 & 130 & \textBF{0.00} & 11 \\
		&  & 3 & Infeas. & 0.01 & 0 & 0.04 & 2781 & 0.01 & 122 & \textBF{0.00} & 11 \\*[0.18cm]
		& 0.5 & 1 & Infeas. & 191.06 & 44255 & 0.05 & 3169 & 0.91 & 7899 & \textBF{0.00} & 31 \\
		&  & 2 & Infeas. & 122.32 & 28105 & 0.03 & 1393 & 0.44 & 5393 & \textBF{0.00} & 39 \\
\bottomrule
	\end{longtable}
\end{landscape}

\end{document}